\documentclass{article}

\usepackage{arxiv}

\usepackage[utf8]{inputenc} 
\usepackage[T1]{fontenc}    
\usepackage{hyperref}       
\usepackage{url}            
\usepackage{booktabs}       
\usepackage{amsfonts}       
\usepackage{nicefrac}       
\usepackage{microtype}      
\usepackage{lipsum}
\usepackage{graphicx}
\graphicspath{ {./images/} }

\usepackage{multirow}%
\usepackage{amsmath,amssymb,amsfonts}%
\usepackage{mathrsfs}%
\usepackage[title]{appendix}%
\usepackage{xcolor}%
\usepackage{textcomp}%
\usepackage{manyfoot}%
\usepackage{booktabs}%
\usepackage{algorithm}%
\usepackage{algorithmicx}%
\usepackage{algpseudocode}%
\usepackage{listings}%
\usepackage{subfig}

\newtheorem{theorem}{Theorem}
\newtheorem{myAss}[theorem]{Assumption}%
\newtheorem{lemma}[theorem]{Lemma}%
\newtheorem{proof}{proof}

\title{Adaptive Moment Estimation Optimization Algorithm Using Projection Gradient for Deep Learning}

\author{
	Yongqi Li \\
	School of Mathematical Sciences\\
	University of Electronic Science and Technology of China\\
	Chengdu, Sichuan, Chain 611731 \\
	\texttt{yongqi2023@126.com} \\
	\And
	Xiaowei Zhang* \\
	School of Mathematical Sciences\\
	University of Electronic Science and Technology of China\\
	Chengdu, Sichuan, Chain 611731 \\
	\texttt{x.w.zhang@126.com} \\
}

\begin{document}
	\maketitle
	\begin{abstract}
		Training deep neural networks is challenging. To accelerate training and enhance performance, we propose PadamP, a novel optimization algorithm. PadamP is derived by applying the adaptive estimation of the $ p $-th power of the second-order moments under scale invariance, enhancing projection adaptability by modifying the projection discrimination condition. It is integrated into Adam-type algorithms, accelerating training, boosting performance, and improving generalization in deep learning. Combining projected gradient benefits with adaptive moment estimation, PadamP tackles unconstrained non-convex problems. Convergence for the non-convex case is analyzed, focusing on the decoupling of $\beta_{1}$ and $\beta_{2}$. Unlike prior work relying on $ \beta_{1}/\sqrt{\beta_{2}} < 1 $, our proof generalizes the convergence theorem, enhancing practicality. Experiments using VGG-16 and ResNet-18 on CIFAR-10 and CIFAR-100 show PadamP's effectiveness, with notable performance on CIFAR-10/100, especially for VGG-16. The results demonstrate that PadamP outperforms existing algorithms in terms of convergence speed and generalization ability, making it a valuable addition to the field of deep learning optimization.
	\end{abstract}


	\section{Introduction}
	First order optimization algorithms have made significant progress in training deep neural networks. One of the most prominent algorithms is stochastic gradient descent (SGD)\cite{ref1} . Despite simple, SGD performs well in many applications. However, its disadvantage is that it scales the gradient uniformly in all directions, which can lead to poor performance and limited training speed when the training data is sparse. In recent years, adaptive variants of SGD have emerged and have been successful due to their convenient automatic learning rate tuning mechanism. Adagrad\cite{ref2} is probably the first along this line of research and significantly outperforms vanilla SGD in sparse gradient scenarios. Despite its initial success, Adagrad has later find to exhibit performance degradation, particularly in cases where the loss function is non-convex or the gradient is dense. Several variants of Adagrad have been developed, including RMSprop\cite{ref3}, Adadelta\cite{ref4}, and Nadam\cite{ref5}.
	
	The Adam\cite{ref6} algorithm proposed by Kingma et al. in 2014. It integrates the best of previous algorithms such as RMSprop and Adadelta's squared accumulation of historical gradients, i.e., controlling the learning rate by exponential moving average. It also adds the idea of momentum and considering the information of the historical gradient. Adam not only accumulates historical gradient information by using exponential moving average, but also corrects for the bias of the calculated iteration direction, the adaptive learning rate, first-order moment estimation and the second-order moment estimation. It is well suited for problems with large-scale data and parameters, as well as for solving problems containing Gaussian noise or sparse gradients. Adam tends to perform well in real-world applications, and is therefore favored by a wide range of researchers and engineers. However, Adam is by no means foolproof, and its convergence property is not always guaranteed. Reddi\cite{ref7} et al. thoroughly proved that Adam does not guarantee convergence even in simple convex optimization problems; Shazeer\cite{ref8} et al. also showed empirically that Adam's parameter updates are not stable and its second moment may be outdated. Luo\cite{ref9} et al. studied Adam's effective update step in training and find that its second moment produces extreme learning rates; Zaheer\cite{ref10} et al. find that Adam's performance may be affected by different values of $\lambda$, which are initially designed to avoid zeros in the denominator. There are other variants of the adaptive gradient method, such as SC-Adagrad/SC-RMSprop\cite{ref11}, which derive logarithmic regret bounds for strongly convex functions. 
	
	The research on Adam has led some scholars to incorporate the projected gradient theory from traditional optimization theory into Adam to improve its optimization effect and enhance the accuracy of deep learning. Madgrad\cite{ref30} improving gradient accumulation by dual averaging with cube root scaling, breaking with the previous tradition of square root gradient accumulation. It lead the first generalised $ p $-power of gradient accumulation.
	And then Chen \& Gu\cite{chen2018closing} propose the Padam algorithm, which replaces the square root operation of second-order moments in Adam with a $ p $-power partially adaptive moment estimation method. It solves the small learning rate dilemma of the adaptive gradient method and allows for faster convergence. The value of $p \in [0,1/2] $ is a partially adaptive parameter, with $ 1/2 $ being the maximum possible value of $ p $. Using a larger $ p $ may cause the proof to not converge. When $ p $ approaches 0, the algorithm reduces to sgdm. When $ p = 1/2 $, the algorithm becomes Amsgrad. Therefore, Padam unifies Amsgrad and SGD. According to AdamP\cite{ref13}, introducing momentum in the GD optimizer results in a faster decrease in the effective step size of the scale-invariant weights. The proposed AdamP eliminates the radial component or direction of the paradigm increase in each optimization step. The method is scale-invariant and only alters the effective step size, preserving the original convergence properties of the GD algorithm.
	
	Applying $ p $-th  powers of gradient accumulation to the application of the AdamP algorithm to design a new PadamP algorithm, our work incorporates ideas from the work referenced in \cite{ref30}, \cite{ref12} and \cite{ref13}. More specifically, the adaptivity of the AdamP algorithm is increased based on the partial adaptation to the second order moments, and the introduction of projection techniques eliminates the radial component or direction of the paradigm growth. The method is scale-invariant, changing only the effective step size. Then we prove the convergence of the proposed algorithm for non-convex situation. Unlike the work in the AdamP, in this paper we use the learning rate multiplied by cosine similarity as a projection discriminant. Experiments show that the proposed algorithm performs better than AdamP and other optimizers.
	
	There are two main views on the proof of convergence: one is based on regret bounds and the other is based on stable points. In convex optimization problems, we usually adopt the regret bound to analyse convergence; while for non-convex problems, we rely more on the analysis of stable points. Our convergence analysis draws on the framework of the work\cite{ref14} and further extends the relevant theorems to the context of unbiased estimation of first-order moments, obtaining in a more generalised convergence theorem. Several experiments were  to verify the effectiveness and performance of our proposed PadamP algorithm.
	
	To sum up, the main contributions can be summarized as follows:
	
	(i) The PadamP is obtained by applying the adaptive estimation of $ p $-th powers of second-order moments in Padam to the AdamP under scale invariance, change the projection discrimination condition to make the projection more adaptive. Also incorporating the PadamP into Adam-type optimization algorithms for deep learning that speed up the training process, improve performance, and enhance the generalisation ability of deep neural networks. The PadamP combines the benefits of the projected gradient algorithm for solving unconstrained non-convex optimization problems with adaptive moment estimation.
	
	(ii) The convergence for non-convex case is analyzed. The convergence analysis
	tackles a hard situation: decoupling of the first-order moment estimation coefficients $\beta_{1}$ and the second-order moment estimation coefficients $\beta_{2}$. A large amount of 
	work have been based on the relationship $\frac{\beta_{1}}{\sqrt{\beta_{2}}} < 1$  between $\beta_{1}$ and $\beta_{2}$, but it have not been possible to decouple. Our proof generalises the convergence theorem, making it more practical.
	
	(iii) Numerical experiments were conducted on image classification tasks using the popular VGG-16 and ResNet-18 networks on the CIFAR-10 and CIFAR-100 datasets, respectively. The experimental results demonstrate the effectiveness of PadamP and provide satisfactory performance on the CIFAR-10/100 datasets, especially for the VGG-16 network.

	\section{Preliminaries}
	\label{sec:headings}
	Here some necessary knowledge are prepared for better understanding.
	
	\subsection{Notation}
	
	Scalars are denoted by lower case letters, vectors by lower case bold face letters, and matrices by upper case bold face letters. For a vector $\mathbf{\theta} \in \mathbb{R}^{d}$, we denote the $\ell_{2}$ norm of $\mathbf{\theta}$ by , the $\ell_{\infty}$ norm of $\mathbf{x}$ by $\|\mathbf{\theta}\|_{2}=\sqrt{\sum_{i=1}^{d} \theta_{i}^{2}}$. For a sequence of vectors $\left\{\mathbf{\theta}_{j}\right\}_{j=1}^{t}$, we denote by $\theta_{j, i}$ the $i$-th element in $\mathbf{\theta}_{j}$. We also denote $\mathbf{\theta}_{1: t, i}=\left[x_{1, i}, \ldots, x_{t, i}\right]^{\top}$. With slight abuse of notation, for two vectors $\mathbf{a}$ and $\mathbf{b}$, we denote $\mathbf{a}^{2}$ as the element-wise square, $\mathbf{a}^{p}$ as the element-wise power operation, $\mathbf{a} / \mathbf{b}$ as the element-wise division and $ \max (\mathbf{a}, \mathbf{b}) $ as the element-wise maximum. We denote by $ \operatorname{diag}(\mathbf{a}) $ a diagonal matrix with di agonal entries $a_{1}, \ldots, a_{d}$. Given two sequences $\left\{a_{n}\right\}$ and $\left\{b_{n}\right\}$, we write $a_{n}=O\left(b_{n}\right) $ if there exists a positive constant $C$ such that $a_{n} \leq C b_{n}$ and $a_{n}=o\left(b_{n}\right) $ if $a_{n} / b_{n} \rightarrow 0 $ as $n \rightarrow \infty$. Notation $\widetilde{O}(\cdot)$ hides logarithmic factors. We also denote the normlization $ \operatorname{Norm}(\cdot) $ satisfy $ \operatorname{Norm}_{\boldsymbol{k}}(\theta)=\frac{\theta-\mu_{\boldsymbol{k}}(\theta)}{\sigma_{\boldsymbol{k}}(\theta)} $, where  $ \mu_{\boldsymbol{k}} $ , $ \sigma_{\boldsymbol{k}} $ are the mean and standard deviation functions
	along the axes $ k $.
	
	\subsection{Stochastic Optimization, Generic Adam and Stationary Point}
	\textbf{Stochastic Optimization}
	For any deep learning or machine learning model, it can be analysed by stochastic 
	optimization framework. In the first place, consider the problem in the following form:
	\begin{equation} 
		{\min _{\theta \in \mathcal{X}} \mathbb{E}_{\pi}[\mathcal{L}(\theta, \pi)]+ \sigma(\theta)},
	\end{equation}
	where $\mathcal{X} \subseteq \mathbb{R}^{d}$ is feasible set. $\pi$ is a 
	random variable with an unknown distribution,
	representing randomly selected data sample or random noise. $\sigma(\theta)$ is a regular term.
	For any given $\theta$, $\mathcal{L}(\theta, \pi)$ usually represents the loss function on sample
	$\pi$. But for most practical cases,  the distribution of $\pi$ can not be obtained. 
	Hence the expectation $\mathbb{E}_{\pi}$ can not be computed. 
	Now there is another one to be considered,
	known as Empirical Risk Minimization Problem(ERM):
	\begin{equation}\label{eq2}
		{\min _{\theta \in \mathcal{X}} f(\theta) = \frac{1}{N} \sum_{i=1}^{N} \mathcal{L}_{i}(\theta)+\sigma(\theta)},
	\end{equation}
	where $\mathcal{L}_{i}(\theta)=\mathcal{L}(\theta, \pi_{i})$, $i=1,2,\ldots,N$ and $\pi_{i}$ are samples.
	For the convenience of our discussion below, without loss of generality, 
	the regular term $\sigma(\theta)$ is ignored.
	
	\textbf{Generic Adam}
	There are many stochastic optimization algorithms for solving ERMs, such as SGD\cite{ref1}, Adadelta\cite{ref4}, RMSprop\cite{ref3}, Adam\cite{ref6}, AdamW\cite{ref15},etc. All of the algorithms listed above are first-order optimisation algorithms and can be described by the following generic Adam algorithm (Alg.\ref{alg:G-Adam}) where $\varphi_{t}:\mathcal{X} \rightarrow S^{d}_{+}$, is an unspecified 'mean' function, typically used for inverse weighted first order moment estimation. Generalised estimates of the first order moment of Adam's 
	are biased. Our study will focus on the generalised Adams moment estimates to unbiased estimates, and we will project the gradient and prove convergence based on the gradient and prove convergence based on it.
	
	\begin{algorithm}
		\caption{Generic Adam}
		\label{alg:G-Adam}
		\begin{algorithmic}
			\State Require: ${x}_{1} \in \mathcal{X}$, ${m}_{0}:={0}$, 
			$(\beta  _{1t})_{t \in \mathcal{T} } \subset [0,1)$.
			
			\While{$t=1$ to $T$}
			\State 
			$g_{t}$ : $noisy \ gradient$
			
			$m_{t}\leftarrow \beta _{1t} {m}_{t-1} + (1-\beta _{1t}) {g}_{t}$
			
			${V}_{t} \leftarrow \varphi_{t} (g_{1},g_{2},\ldots,g_{t})$
			
			$\theta_{t+1} \leftarrow \theta_{t}-\frac{\alpha_{t}}{\sqrt{V_{t}+\epsilon I}} \cdot m_{t}$
			\EndWhile
		\end{algorithmic}
	\end{algorithm}
	
	\textbf{Stationary Point}
	A stationary point of a differentiable function $f(\theta)$ occurs when $\| \nabla f(\theta^{*}) \|_{2} = 0$, where $ \nabla f(\theta^{*}) $ denotes the gradient of $f$. In the case of convex functions,$\theta^{*}$ where $\theta^{*} \in \mathcal{X}$ serve as a global minimizer of $f$ over $\mathcal{X}$. However, in practical scenarios, $f$ is often non-convex. When an optimization algorithm yields a stationary point solution, the gradient of $f$ is nearly zero in the vicinity of the solution, resulting in apparent stagnation in the optimization process. Although prolonged runs may escape from such local optima, the associated time costs could be deemed unacceptable.

	\section{PadamP}
	\label{}
	\subsection{Proposed Algorithm}

	In the realm of deep learning optimization, developing efficient and effective algorithms is of paramount importance. This section delves into the preparation and design of a novel optimization algorithm, PadamP, by drawing insights from existing algorithms such as Madgrad\cite{ref30}, Padam\cite{ref12}, and AdamP\cite{ref13}. The work towards the creation of PadamP is a sequential process, with each step building upon the previous one, aiming to address the limitations of traditional optimization methods.
	
	\textbf{Motivation and Background}
	Deep learning optimization has witnessed the emergence of numerous algorithms, each striving to achieve better performance. Adaptive gradient methods, like Adam, have shown promise in quickly adjusting the learning rate but often suffer from generalization issues compared to SGD with momentum. This gap in generalization performance has spurred researchers to seek solutions that can combine the best of both worlds.
	
	Madgrad, proposed by Aaron Defazio et al., represents a significant advancement in this field. It is based on the dual averaging formulation (DA) of AdaGrad\cite{ref2}, which has the following general form:
	\begin{equation}\label{eq10}
		\begin{array}{l}
			g_{t} =\nabla f\left(\theta_{t}, \xi_{t}\right), \\
			s_{t+1} =  s_{t}+\eta_{t} g_{t}, \\
			\theta_{t+1} =\arg \min _{\theta}\left\{\left\langle s_{t+1}, \theta\right\rangle+\beta_{t+1} \psi(\theta)\right\},
		\end{array}
	\end{equation}
	Here, $\psi(\theta)$ is the Bregman divergence core composition function. 
	The gradient buffer $ s_{0} $ is set to the zero vector as the initial value. The most basic form of dual averaging comes into play when the standard Euclidean squared norm is employed. Specifically, we have $ \psi(\theta)=\frac{1}{2}\left\|\theta-\theta_{0}\right\|^{2} $ and $ \lambda_{t}=1 $. Under such circumstances, the method can be expressed as follows:
	\begin{equation}\label{eq12}
		\theta_{t + 1}=\theta_{0}-\frac{1}{\beta_{t + 1}}\sum_{i = 0}^{t}g_{i}.
	\end{equation}
	
	When the objective is non-smooth, stochastic, or both,  $ \beta $  sequences like  $ \beta_{t+1}=\sqrt{t+1} $  lead to convergence. Equation \ref{eq12} seems unlike SGD, but SGD's update  $ \theta_{t+1}=\theta_{t}-\gamma_{t} \nabla f\left(\theta_{t}, \xi_{t}\right) $  can be rewritten as  $ \theta_{t+1}=\theta_{0}-\sum_{i=0}^{t} \gamma_{i} g_{i} $ . For convergence without a fixed stop,  $ \gamma_{i} \propto 1 / \sqrt{i+1} $  is typical.
	
	Comparing SGD and DA at step  $ t $ , SGD weights newer  $ g_{i} $  less than earlier ones in the sum, while DA weights all  $ g_{i} $  equally. This difference explains why DA methods and SGD behave differently in practice, even without extra regularization or non-Euclidean functions.
	
	On benchmarks, DA trails SGD. AdaGrad and DA underperform due to $\frac{1}{\sqrt{i + 1}}  $ learning rate sequence, which is harmful. Sqrt-decay causes quick convergence to a poor local minima. AdaGrad and AdaGrad-DA have an implicitly decreasing sequence, with $\frac{1}{\sqrt{t}}$ rate when gradients are similar, not ideal. Using the same learning rate scheme makes scholars think AdaGrad is less effective than Adam. For DA, we propose $\lambda_{i}=(i + 1)^{1 / 2} \gamma_{i}$. This gives the leading term in 
	$ 	\theta_{t+1} =\theta_{0}-\frac{1}{\sqrt{t+1}} \sum_{i=0}^{t} \lambda_{i} g_{i} $
	constant weight across k, like SGD's constant step size and keeping past gradients' $\sqrt{t + 1}$ decay. It significantly improves DA's test-set performance with SGD's schedule.
	
	Next, the goal is to integrate dual averaging, and AdaGrad's adaptivity. However, combining the coordinate-wise adaptivity of AdaGrad with the weighted gradient sequence $\lambda_{i}=\sqrt{i + 1}$ is challenging as different denominator-weighting methods have flaws in maintaining the effective step-size. Madgrad firstly use a cube-root denominator in its update formula
	\begin{equation}
		\theta_{t + 1}=\theta_{0}-\frac{1}{\sqrt[3]{\sum_{i = 0}^{t}(i + 1)^{1/2}\gamma_{i}g_{i}^{2}}}\sum_{i = 0}^{t}(i + 1)^{1/2}\gamma_{i}g_{i}.
	\end{equation}
	
	This is based on the solution of the minimization problem 
	\begin{equation}
		\begin{aligned}
			\min_{s}&\sum_{i = 0}^{t}\sum_{d = 0}^{D}\frac{g_{id}^{2}}{s_{d}}\\
			s.t.&\\
			&\|s\|_{2}^{2}\leq c,\\
			&s_{d}>0, \forall d,
		\end{aligned}
	\end{equation}
	with an L2 norm penalty. The solution to this problem is $s_{d}\propto\sqrt[3]{\sum_{i = 0}^{t}g_{id}^{2}}$. Given that $\sum_{i}^{t}(i + 1)^{1/2}\propto(t + 1)^{3/2}$, the cube-root operation generates a $\sqrt{t + 1}$ scaled denominator, which can counteract the square-root growth of $\lambda$.

	%
	%
	%
	
	\textbf{Partially Adaptive Momentum Estimation} 
	Despite the success of Madgrad in certain aspects, there is still room for improvement. Adaptive gradient methods like Adam, in the form $\theta_{t}:=\theta_{t}-\eta_{t}\hat{m}_{t}/\sqrt{\hat{v}_{t}}$ ($\hat{m}_{t}$ being the update direction), often lack the generalization and stability of SGD with momentum during training, although they may exhibit better training performance initially.
	
	Drawing inspiration from the design of Madgrad, we introduce the concept of the $p$-th power. Our goal is to create an algorithm that seamlessly combines the characteristics of Adam and SGD-Momentum. We propose the following update rule:
	\begin{equation}
		\theta_{t}:=\theta_{t}-\eta_{t}\hat{m}_{t}/(\hat{v}_{t})^{p},
	\end{equation}
	where $p\in[0,1/2]$. When $p$ approaches $0$, the algorithm simplifies to SGD with momentum. In this case, the influence of the adaptive component $(\hat{v}_{t})^{p}$ diminishes, and the update rule resembles that of SGD with momentum. When $p = 1/2$, the algorithm becomes similar to Adam. This generalization allows for a smooth transition between different optimization behaviors, enabling the algorithm to adapt to various scenarios.
	
	\textbf{Normalization Layer and Scale Invariance}
	Normalization layers play a impotant role in modern deep learning. For a function $g(\theta)$, if $g(c\theta)=g(\theta)$ for any $c > 0$, then $g$ is said to be scale-invariant. In the context of deep learning, normalization operations such as batch normalization (BN), layer normalization (LN), instance normalization (IN), and group normalization (GN) result in scale-invariant weights. Specifically,
	\begin{equation}
		\text{Norm}(\boldsymbol{w}^{\top}\theta)=\text{Norm}((c\boldsymbol{w})^{\top}\theta),
	\end{equation}
	for any constant $c > 0$.
	
	This implies that the weights $\theta$ preceding the normalization layer are scale-invariant.
	
	The $\ell_{2}$-norm of these scale-invariant weights, $\|\theta\|_{2}$, does not affect the forward computation $f_{\boldsymbol{\theta}}$ or the backward gradient $\nabla_{\theta}f(\theta)$ of a neural network layer $f_{\theta}$ parameterized by $\theta$. These scale-invariant weights can be represented as $\ell_{2}$-normalized vectors $\widehat{\theta}:=\frac{\theta}{\|\boldsymbol{\theta}\|_{2}}\in\mathbb{S}^{d-1}$ (i.e., $c=\frac{1}{\|\boldsymbol{\theta}\|_{2}}$). Understanding this scale-invariance property is essential as it has implications for the behavior of optimization algorithms.
	\begin{figure}[htpb]
		\centering
		\includegraphics[width=0.5\linewidth]{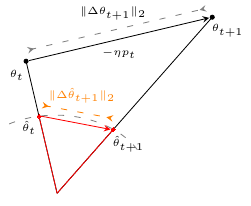}
		\caption{Normalization layer and scale invariance}
		\label{fig1}
	\end{figure}\\
	
	\textbf{Decay of Effective Step Sizes}
	Momentum is a widely used technique in gradient-based optimization, enhancing the convergence speed by enabling the parameter vector $\theta$ to escape high-curvature regions and handle noisy or vanishing gradients. Mathematically, the momentum update rule is given by:
	\begin{equation}\label{eq13}
		\theta_{t+1} = \theta_{t}-\eta p_{t}, \quad p_{t} = \beta p_{t-1}+\nabla_{\theta_{t}} f\left(\theta_{t}\right).
	\end{equation}
	
	When considering the norm growth of the parameters during optimization, we have the following equations for SGD and SGD with momentum:
	\begin{equation}\label{eq14}
		\left\|\theta_{t+1}\right\|_{2}^{2}=\left\|\theta_{t}\right\|_{2}^{2}+\eta^{2}\left\|p_{t}\right\|_{2}^{2},
	\end{equation}
	\begin{equation}\label{eq15}
		\left\|\theta_{t+1}\right\|_{2}^{2}=\left\|\theta_{t}\right\|_{2}^{2}+\eta^{2}\left\|p_{t}\right\|_{2}^{2}+2 \eta^{2} \sum_{k=0}^{t-1} \beta^{t-k}\left\|p_{k}\right\|_{2}^{2}.
	\end{equation}
	
	By comparing these two equations, we can observe that the equation for GD with momentum has an additional non-negative term $2\eta^{2}\sum_{k = 0}^{t-1}\beta^{t-k}\|p_{k}\|_{2}^{2}$ on the right-hand side. This term accumulates past updates and significantly accelerates the increase in weight norms when momentum is used.
	
	To quantify this difference, we have Lemma \ref{lem1}:
	\begin{lemma}\label{lem1}
		Let  $ \|\theta_{t}^{\mathrm{GD}}\|_{2} $  and  $ \|\theta_{t}^{\mathrm{GDM}}\|_{2} $  be the weight norms at step  $ t \geq 0 $ , following the recursive formula in Equation \ref{eq14} and Equation \ref{eq15}, respectively. We assume that the norms of the updates  $ \left\|p_{t}\right\|_{2} $  for GD with and without momentum are identical for every  $ t \geq  0 $ . We further assume that the sum of the update norms is non-zero and bounded: $  0<\sum_{t \geq 0}\left\|p_{t}\right\|_{2}^{2}<   \infty  $. Then, the asymptotic ratio between the two norms is given by:
		\begin{equation}\label{eq16}
			\frac{\|\theta_{t}^{\mathrm{GDM}}\|_{2}^{2}-\left\|\theta_{0}\right\|_{2}^{2}}{\left\|\theta_{t}^{\mathrm{GD}}\right\|_{2}^{2}-\left\|\theta_{0}\right\|_{2}^{2}} \longrightarrow 1+\frac{2 \beta}{1-\beta} \quad \text { as } \quad t \rightarrow \infty.
		\end{equation}
	\end{lemma}
	
	Proof for the reference\cite{ref13}. 
	
	This shows that the use of momentum can lead to a much faster increase in weight norms, which may have a negative impact on the effective step size and, consequently, the convergence of the optimization algorithm.
	
	To address the issues associated with the decay of effective step sizes and the need to combine the advantages of different optimization methods, we propose the PadamP algorithm. We start with the partially adaptive momentum estimation concept and incorporate the scale-invariance considerations.
	
	We retain the benefits of momentum while eliminating the accumulated error term that causes the rapid increase in weight norms. Let $\Pi_{\theta}(\cdot)$ denote the projection onto the tangent space of $\theta$, defined as
	\begin{equation}
		\Pi_{\theta}(\theta):=\theta-\left\langle\widehat{w},\theta\right\rangle\widehat{\theta},
	\end{equation}
	where $\cos(a, b):=\frac{|a^{\top}b|}{\|a\|\|b\|}$ is the cosine similarity. Rather than manually identifying weights before normalization layers, our algorithm automatically detects scale invariances using cosine similarity. We have found that $\delta = 0.1$ is an appropriate value. It is small enough to accurately detect orthogonality, which is crucial for identifying scale-invariant weights, and large enough to capture all such weights.
	\begin{figure}[htpb]
		\centering
		\includegraphics[width=0.55\linewidth]{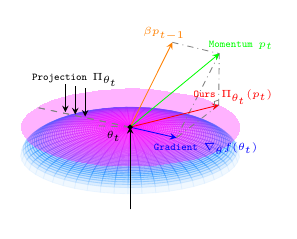}
		\caption{Vector directions of the
			gradient, momentum, and ours.}
		\label{fig2}
	\end{figure}

	The full PadamP algorithm(Alg.\ref{alg:PadamP}) is as follows:
	\begin{algorithm}
		\caption{PadamP}\label{alg:PadamP}
		\begin{algorithmic}[1]
			\Require Learning rate  $\eta_{t}>0$, momentum $0<\beta_{1},\beta_{2}<1$, thresholds $\delta$, $\epsilon>0$, $p\in(0,\frac{1}{2}]$.
			\While{$\boldsymbol{\theta}_{t}$ not converged}
			\State $\boldsymbol{m}_{t}\leftarrow\beta_{1}\boldsymbol{m}_{t-1}+(1-\beta_{1})\nabla_{\boldsymbol{\theta}}f_{t}(\boldsymbol{\theta}_{t})$
			\State $\boldsymbol{v}_{t}\leftarrow\beta_{2}\boldsymbol{v}_{t-1}+(1-\beta_{2})(\nabla_{\boldsymbol{\theta}}f_{t}(\boldsymbol{\theta}_{t}))^{2}$
			\State $\hat{\boldsymbol{m}}_{t}\leftarrow\boldsymbol{m}_{t}/(1-\beta_{1}^{t})$
			\State $\hat{\boldsymbol{v}}_{t}\leftarrow\boldsymbol{v}_{t}/(1-\beta_{2}^{t})$
			\State $\boldsymbol{p}_{t}\leftarrow\hat{\boldsymbol{m}}_{t}/(\hat{\boldsymbol{v}}_{t})^{p}$
			\State $\boldsymbol{q}_{t}\leftarrow\begin{cases}\Pi_{\boldsymbol{\theta}_{t}}(\boldsymbol{p}_{t})&\text{if }\cos(\boldsymbol{\theta}_{t},\nabla_{\boldsymbol{\theta}}f(\boldsymbol{\theta}_{t}))<\delta\eta_{t}/\sqrt{\text{dim}(\boldsymbol{\theta})}\\\boldsymbol{p}_{t}&\text{otherwise}\end{cases}$
			\State $\boldsymbol{\theta}_{t + 1}\leftarrow\boldsymbol{\theta}_{t}-\eta_{t}\boldsymbol{q}_{t}$
			\EndWhile
		\end{algorithmic}
	\end{algorithm}

	
	In our algorithm, the update direction $\boldsymbol{p}_{t}$ is adjusted based on the $p$-th power of the second-order moment $\hat{\boldsymbol{v}}_{t}$, and the projection operation $\Pi_{\boldsymbol{\theta}_{t}}(\cdot)$ is applied to control the growth of the weight norms. This combination allows PadamP to adaptively adjust the learning rate while maintaining the stability and generalization performance similar to SGD with momentum.
	
	PadamP changes the projection discrimination condition by adding a learning rate $\eta_{t}$, which will make the projection discrimination condition more and more stringent, and the projection is less and less necessary at the later stages of the update when the objective function is getting closer and closer to the optimal point, because at this point the momentum leads to a smaller and smaller paradigm, which reduces the amount of computation and speeds up the rate of convergence.
\subsection{Assumptions and Convergence Analysis}
Based on the demanding of our convergence analysis, 
the assumptions and lemmas are directly listed as follows.
\begin{myAss}\label{ass31}
	There exists a constant $C_{1}$, for all $t \in \mathcal{T}$,  
	$\| \nabla f(\theta_{t}) \| \leqslant C_{1}$, $\| g_{t} \| \leqslant C_{1}$.
	%
\end{myAss}
\begin{myAss}\label{ass32}
	$ f(\theta)=\mathbb{E}_{\xi} f(\theta ; \xi) $  is  L  smooth: for any  $ \theta_{1} $, $ \theta_{2} \in \mathbb{R}^{d} $ , it satisfied that  $ \mid f(\theta_{1})-f(\theta_{2})-   \langle\nabla f(\theta_{2}), \theta_{1}-\theta_{2}\rangle \left\lvert\, \leq \frac{L}{2}\|\theta_{1}-\theta_{2}\|_{2}^{2}\right.  $
	
	Assumption \ref{ass32} is frequently used in analysis of gradientbased algorithms. It is equivalent to the  L -gradient Lipschitz condition, which is often written as  $ \|\nabla f(\theta_{1})-\nabla f(\theta_{2})\|_{2} \leq   L\|\theta_{1}-\theta_{2}\|_{2} $. 
\end{myAss}

\begin{myAss}\label{ass33}
	The learning rate $ \eta _t $ is a non-increasing positive term series and satisfies the following conditions:
	\begin{equation}\label{eq18}
		\begin{aligned}
			\sum\limits_{n = 0}^\infty  {{\eta_{t}}}  = \infty ,\\
			\sum\limits_{n = 0}^\infty  {\eta_{t}^2}  < \infty .
		\end{aligned}
	\end{equation}
\end{myAss}
\begin{myAss}\label{ass34}
	$ f $ is lower bounded: $ f\left(\theta^{*}\right)=f^{*}>-\infty $ ,where $ \theta^{*} $ is an optimal solution.
\end{myAss}
\begin{myAss}\label{ass35}
	The noisy gradient $ g_{t} $  is unbiased and the noise is independent: $ g_{t}=\nabla f\left(x_{t}\right)+\zeta_{t}, \mathbb{E}\left[\zeta_{t}\right]=0, $ and $ \zeta_{i} $ is independent of $ \zeta_{j} $ if $ i=j $.
\end{myAss}

Next we provide the main convergence rate result for our proposed algorithm. The detailed proof can be find in the longer version of this paper.

\begin{theorem}\label{th3.1}
	For PadamP, if $ f(\theta) $ satisfies Assumption \ref{ass32} and has a lower bound on its domain of definition, the stochastic gradient $ g_{t} $ satisfies Assumption \ref{ass31}, and the learning rate $\eta_{t}  \le 1$ and satisfies Assumption \ref{ass33}, $ \beta_{1 t}=\beta_{1} \lambda^{t-1}, 0 < \lambda < 1 $, then there are
	\begin{equation}\label{eq20}
		\lim _{T \rightarrow \infty} \min _{t=1: T} \mathbb{E}\left[\left\|g_{t}\right\|^{2}\right] = 0.
	\end{equation}

\end{theorem}

	\section{Experiments}
	\label{}
	In this section, we conduct several experiments of image classification task
	on the benchmark dataset CIFAR-10 and CIFAR-100\cite{ref26} using popular 
	network VGG-16\cite{ref27} and ResNet-18\cite{ref28} to numerically compare 
	PadamP with other optimizers.
	Specifically, VGG-16 has trained on CIFAR-10 and ResNet-18 has trained on CIFAR-100.
	
	The CIFAR-10 and CIFAR-100 dataset both consist of 60000 colour images 
	with 10 and 100 classes respectively. Both two datasets are 
	divided into training and testing datasets, which includes 50000 training images and 
	10000 test images respectively.
	The test batch contains exactly 100 or 1000 randomly-selected images from each class.
	
	VGG-16, proposed by Karen Simonyan\cite{ref27} in 2014, consisting of 16 
	convolution layers and 3 fully-connected layers, is a concise network structure 
	made up of 5 vgg-blocks with consecutive 3$\times$3 convolutional kernels in each vgg-block.
	The final fully-connected layer is 1000-way with a softmax function.
	The ResNet-18 network, proposed by He et al.\cite{ref28}, incorporates residual unit through a short-cut connection, 
	enhancing the network's learning ability. ResNet-18 is organized as a
	7$\times$7 convolution layer, four convolution-blocks including total 32 convolution layers 
	with 3$\times$3 convolutional kernels and finally a 1000-way-fully-connected layer 
	with a softmax function. VGG-16 and ResNet-18 have a strong ability to extract image features.
	All networks have trained for 200 epochs on NVIDIA Tesla V100 GPU. The code and relevant materials for the experiment can be downloaded from https://github.com/zhang-xiaowei/Code.
	
	The default parameter settings for each optimizer are shown in the table \ref{tab:optimizer_params}. The cross entropy is used as the loss function for all experiments for 
	the reason of commonly used strategy in image classification.
	Besides, batch size is set 128 as default.In addition to that, we do some processing on the image, including random crop (RandomCrop(32, padding=4)), RandomHorizontalFlip(), ToTensor() operation, and normalisation operation (Normalize((0.4914, 0.4822, 0.4465), (0.2023, 0.1994, 0.2010))).
	
	\begin{table}[htbp]
		\caption{optimizer hyperparameters}
		\centering
		\begin{tabular}{l|ccccc}
			\hline
			Optimizer & PadamP & AdamP & Adam & AMSgrad & sgdm \\
			\hline
			Initial lr & 1e-3 & 1e-3 & 1e-3 & 1e-3 & 0.1 \\
			beta1 & 0.9 & 0.9 & 0.9 & 0.9 &-\\
			beta2 & 0.999 & 0.999 & 0.99 & 0.99 &-\\
			weight decay & 1e-2 & 1e-2 & 1e-4 & 1e-4 & 5e-4 \\
			momentum &-&-&-&-& 0.9 \\
			\hline
		\end{tabular}
		\label{tab:optimizer_params}
	\end{table}
	\subsection{Selection of the optimal learning rate}
	We first choose the optimal learning rate setting for the PadamP algorithm. All the algorithms are run under different learning rates 
	$\eta_{t} \in \{1e-1,5e-2,1e-2,5e-3,1e-3,5e-4,1e-4\}, t \in \mathcal{T}$, and decay the learning rate
	by $  0.1  $ at the $ 50 $th, $ 100 $th and $ 150 $th epoch.
	Figure \ref{fig3}
	shows that we can choose a learning rate of $ 1e-3 $ as our default setting for the learning rate.
	\begin{figure}
		\centering
		\begin{minipage}[t]{1\linewidth}
			\centering
			\subfloat[Train Loss]{\includegraphics[width= 7cm,height=4.5cm]{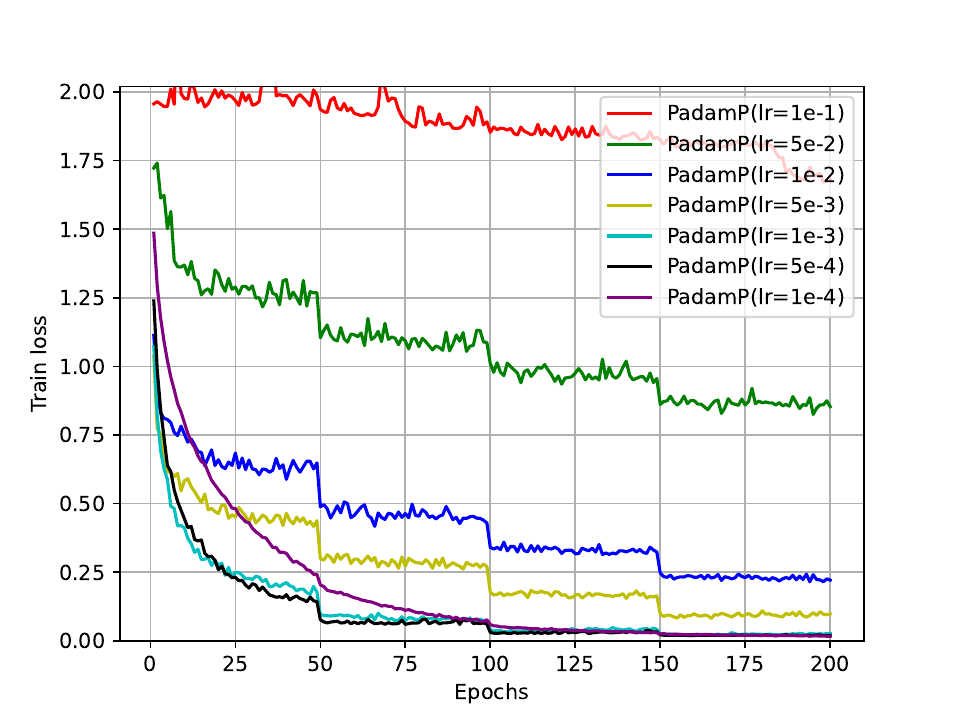}}\\
			\subfloat[Train Accuracy]{\includegraphics[width= 7cm,height=4.5cm]{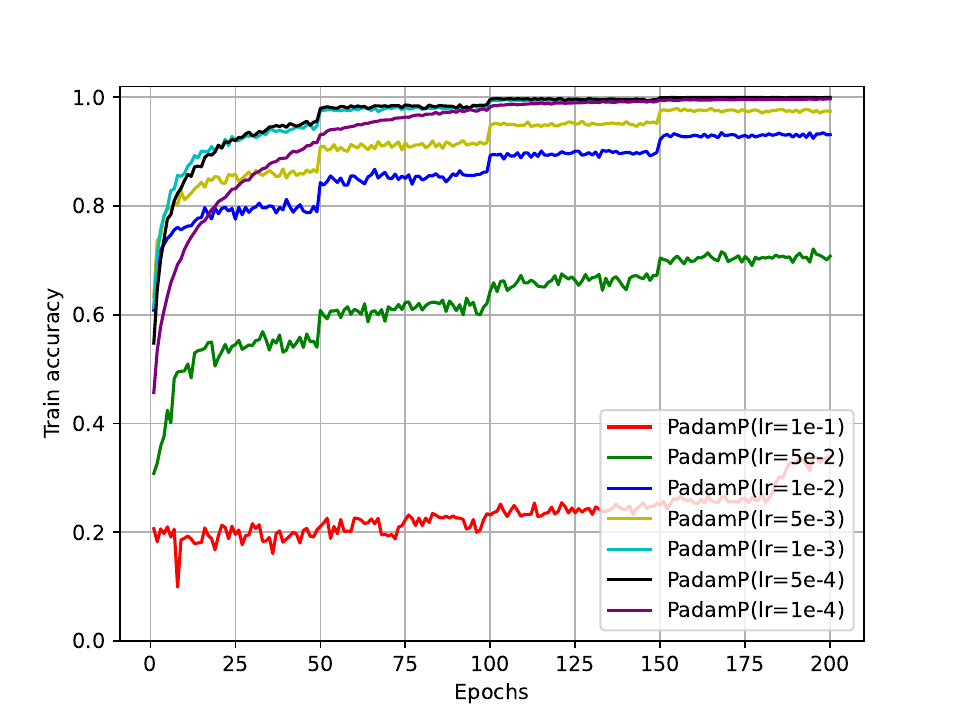}}\\
			\subfloat[Test Accuracy]{\includegraphics[width= 7cm,height=4.5cm]{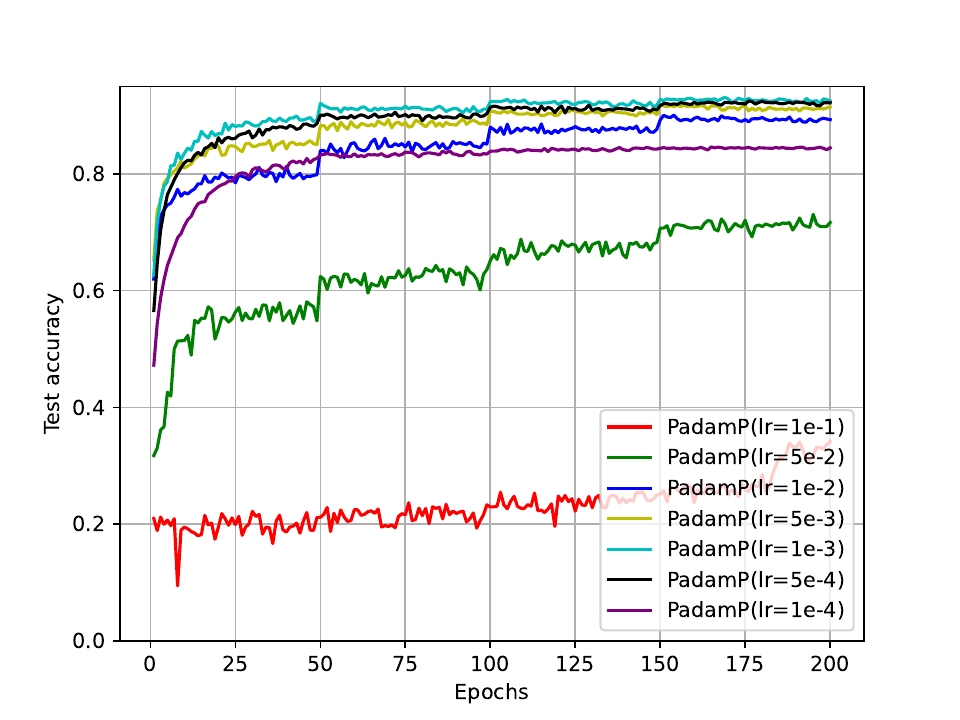}}
			\caption{PadamP under different learning rates. (train VGG-16 on CIFAR-10)}
			\label{fig3}
		\end{minipage}%
	\end{figure}
	

	\subsection{Compare PadamP with other optimizers}
	We firstly compare PadamP with oneself. All the algorithms are run under different $ p \in \{1/4,1/5,1/8\}, \forall t \in \mathcal{T}$.
	From the figure \ref{fig4}-\ref{fig7}, when we choose the VGG-16 network, the optimal $ p $-value can be chosen as $ 1/5 $. And when choose ResNet-18 network, the optimal $ p $-value can be chosen as $ 1/4 $.
	
	\begin{figure}
		\begin{minipage}[t]{0.5\linewidth}
			\centering
			\subfloat[Train Loss]{\includegraphics[width= 7cm,height=4.5cm]{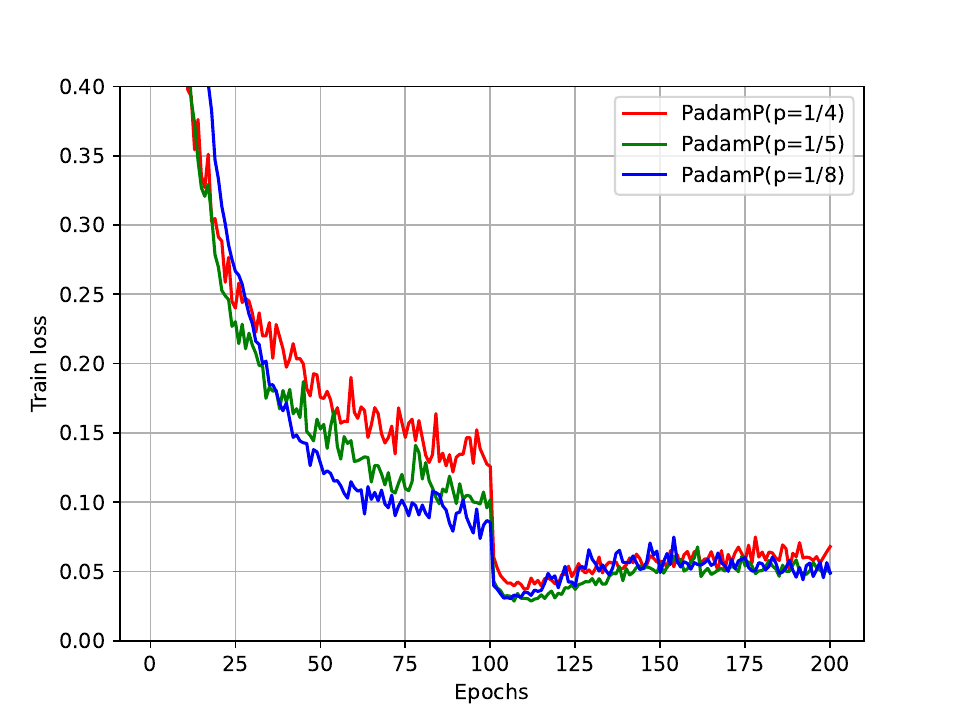}}\\
			\subfloat[Train Accuracy]{\includegraphics[width= 7cm,height=4.5cm]{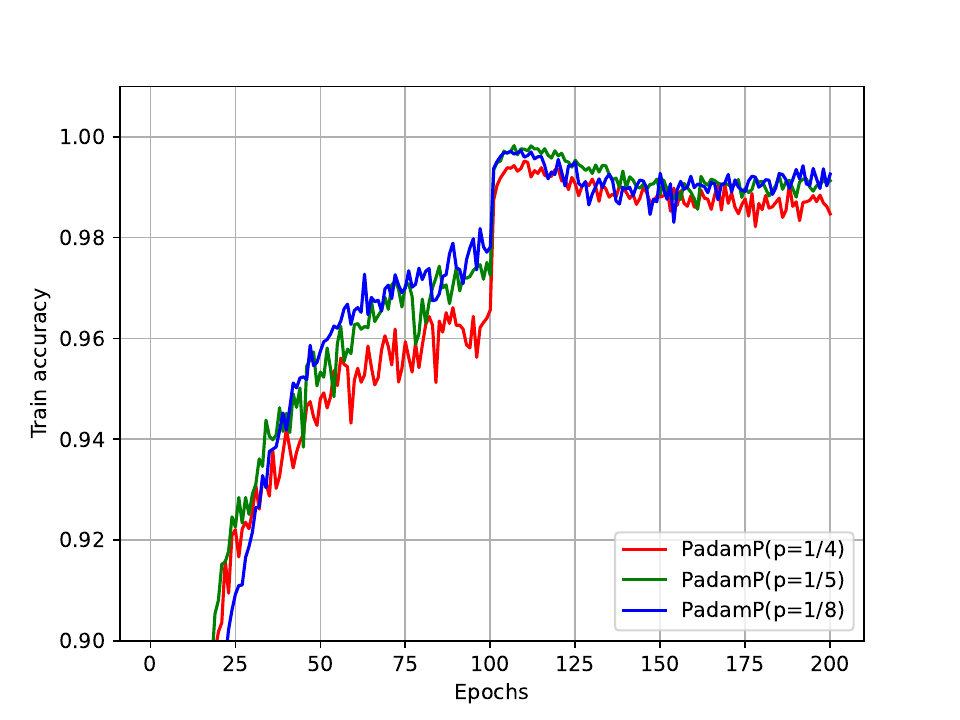}}\\
			\subfloat[Test Accuracy]{\includegraphics[width= 7cm,height=4.5cm]{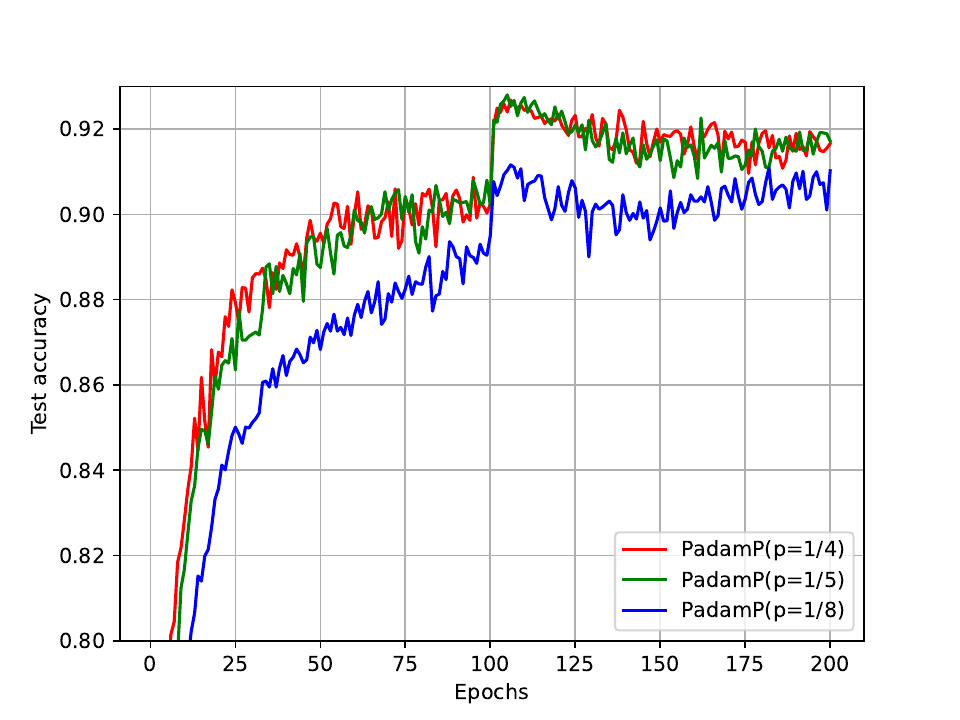}}
			\caption{Different p-values trained on VGG-16 network, CIFAR-10 data(learning rate not decay).}
			\label{fig4}
		\end{minipage}%
		\hspace{0.5cm}
		\begin{minipage}[t]{0.5\linewidth}
			\centering
			\subfloat[Train Loss]{\includegraphics[width= 7cm,height=4.5cm]{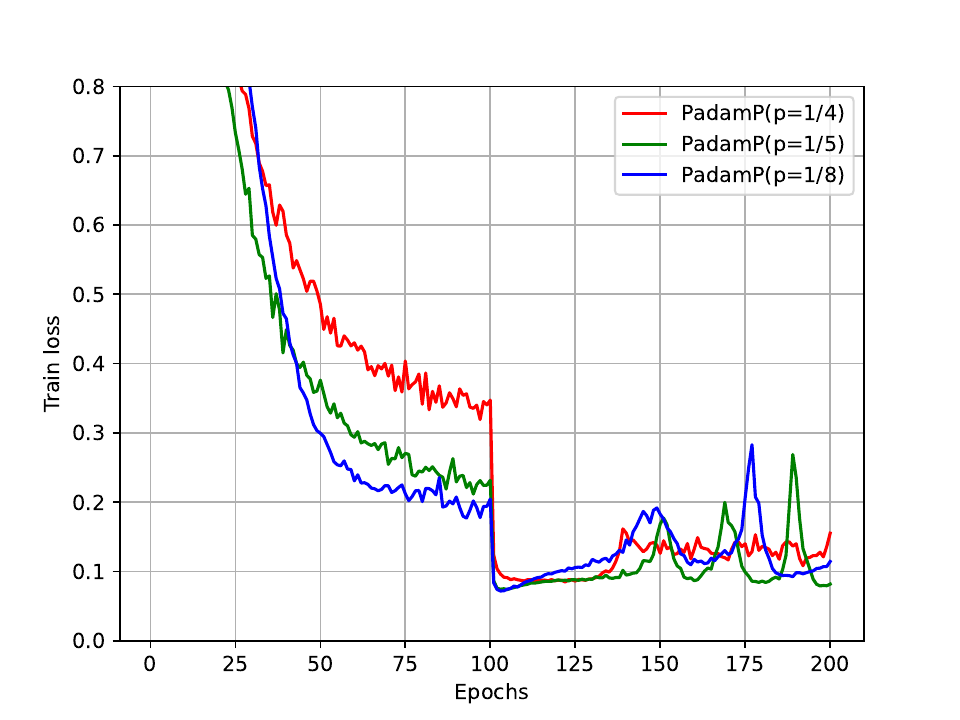}}\\
			\subfloat[Train Accuracy]{\includegraphics[width= 7cm,height=4.5cm]{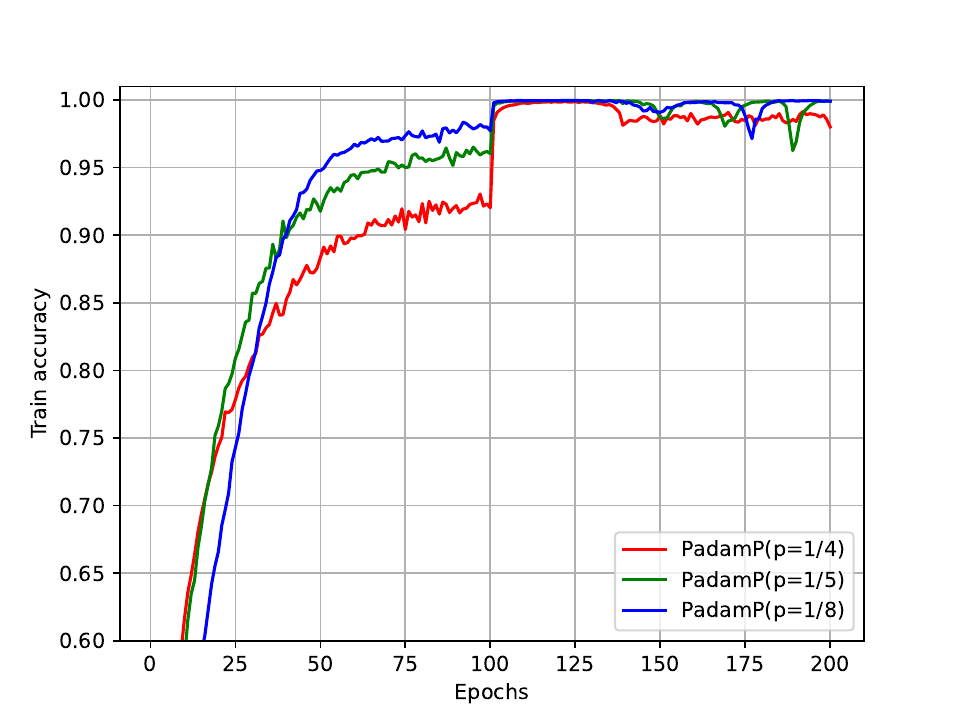}}\\
			\subfloat[Test Accuracy]{\includegraphics[width= 7cm,height=4.5cm]{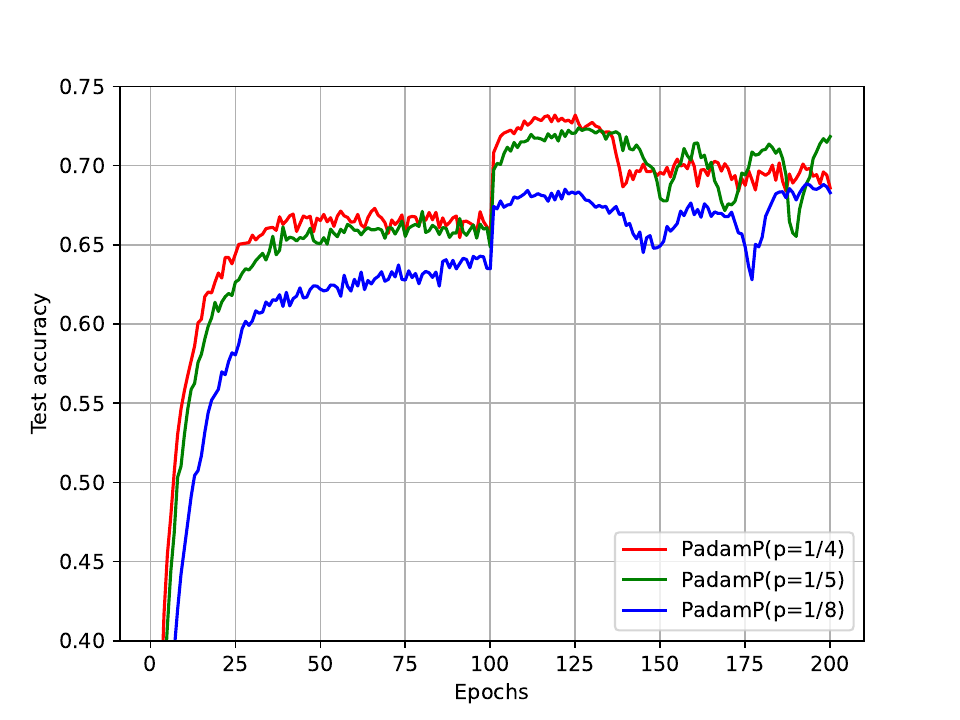}}
			\caption{Different p-values trained on VGG-16 network, CIFAR-100 data(learning rate not decay)}
			\label{fig5}
		\end{minipage}
	\end{figure}
	
	\begin{figure}
		\begin{minipage}[t]{0.5\linewidth}
			\centering
			\subfloat[Train Loss]{\includegraphics[width= 7cm,height=4.5cm]{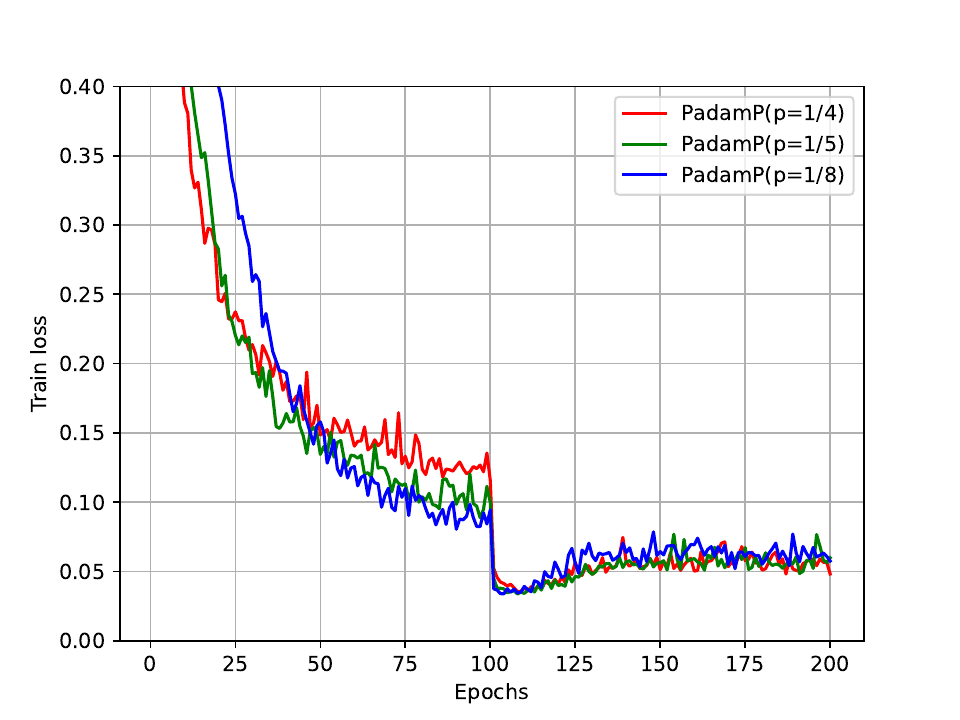}}\\
			\subfloat[Train Accuracy]{\includegraphics[width= 7cm,height=4.5cm]{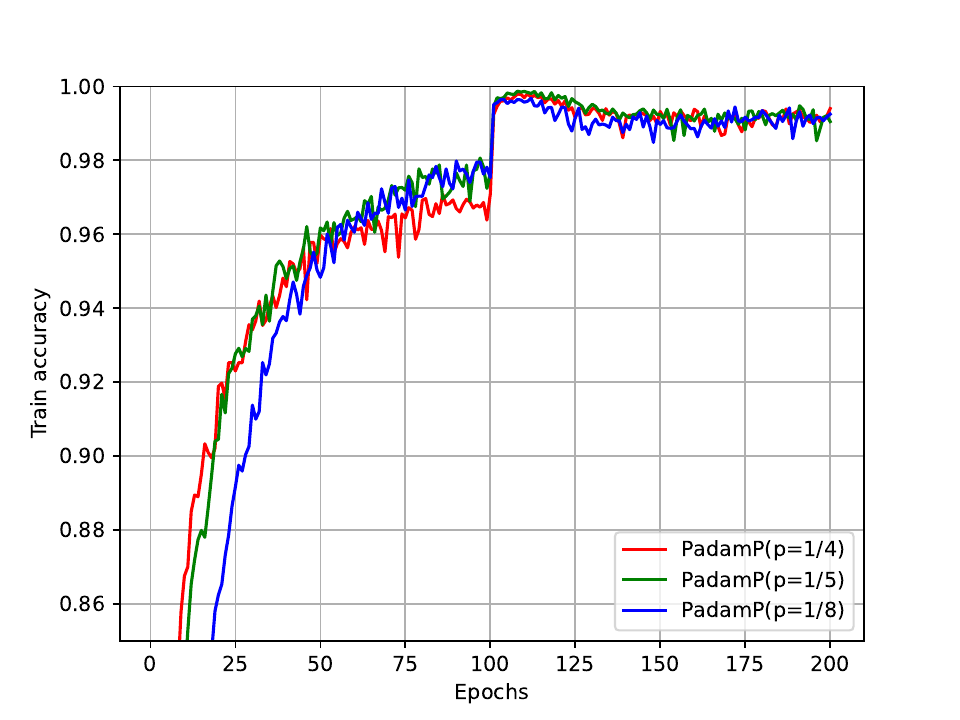}}\\
			\subfloat[Test Accuracy]{\includegraphics[width= 7cm,height=4.5cm]{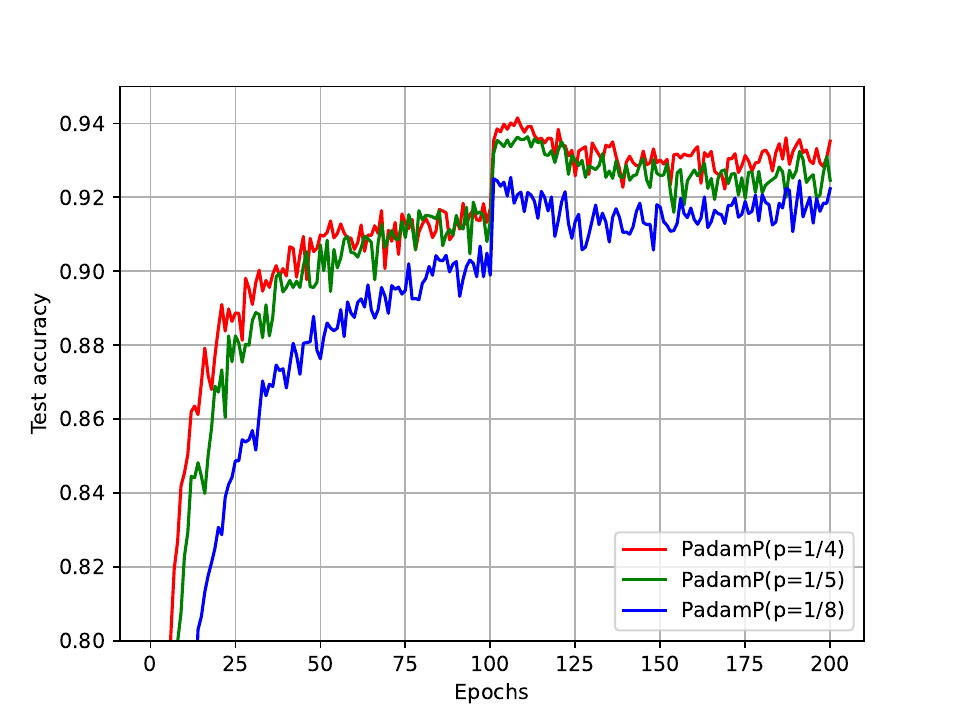}}
			\caption{Different p-values trained on ResNet-18 network, CIFAR-10 data (learning rate not decay).}
			\label{fig6}
		\end{minipage}%
		\hspace{0.5cm}
		\begin{minipage}[t]{0.5\linewidth}
			\centering
			\subfloat[Train Loss]{\includegraphics[width= 7cm,height=4.5cm]{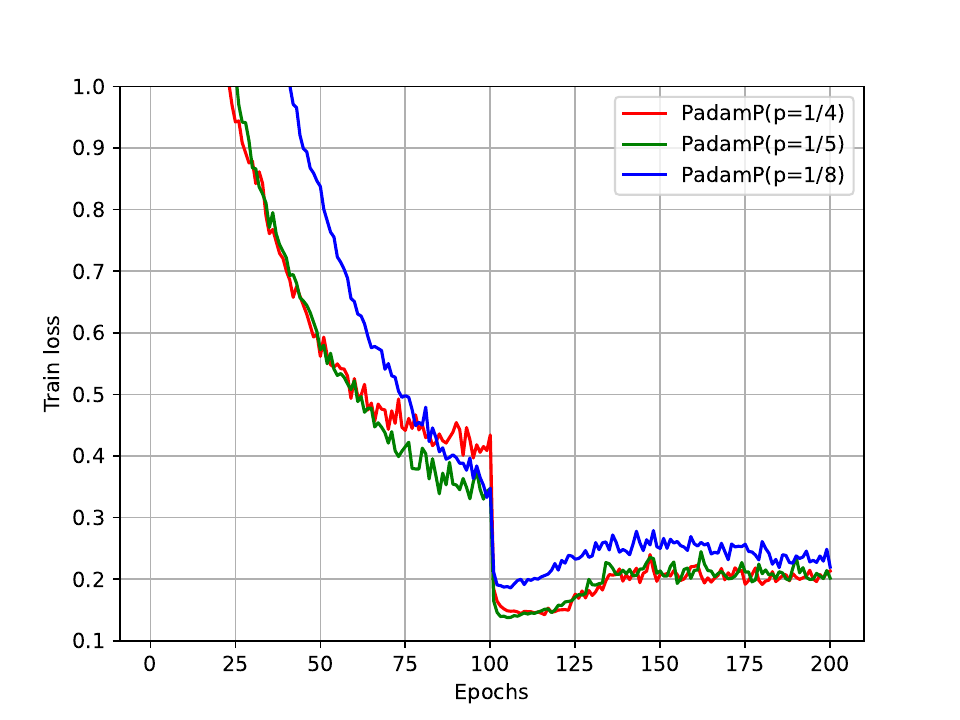}}\\
			\subfloat[Train Accuracy]{\includegraphics[width= 7cm,height=4.5cm]{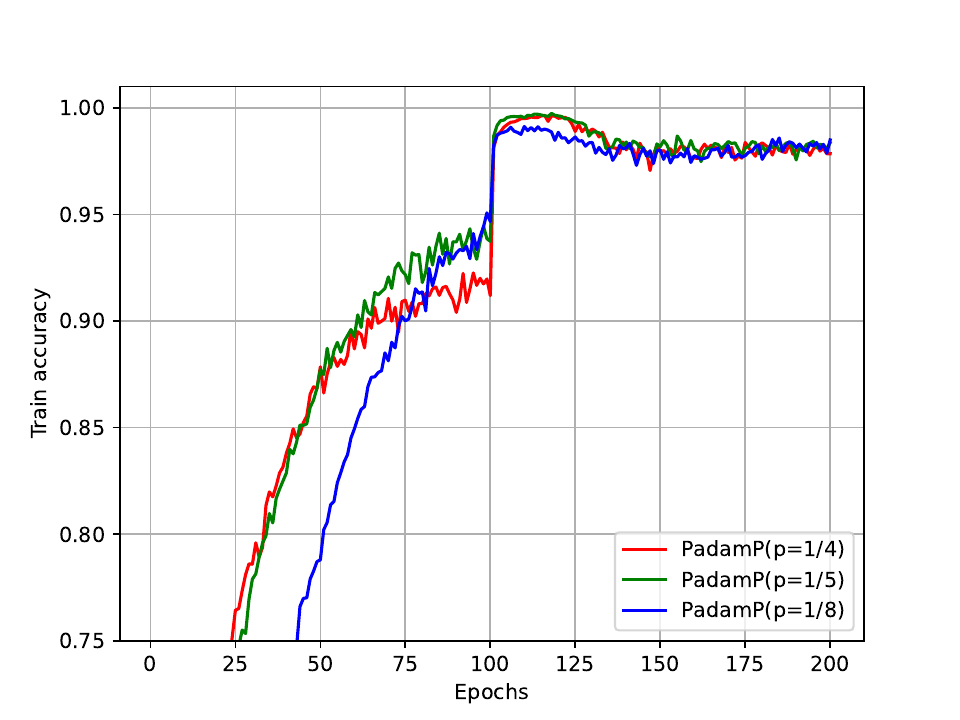}}\\
			\subfloat[Test Accuracy]{\includegraphics[width= 7cm,height=4.5cm]{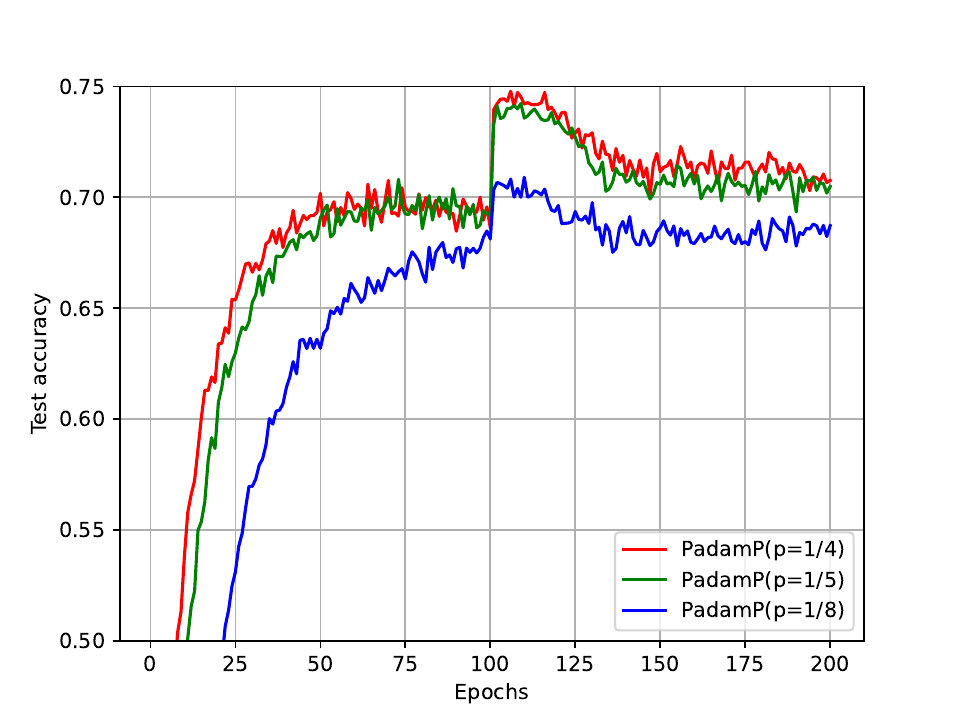}}
			\caption{Different p-values trained on ResNet-18 network, CIFAR-100 data (learning rate not decay)}
			\label{fig7}
		\end{minipage}
	\end{figure}

	In the initial learning rate decay strategy used in most of the work \cite{ref32}\cite{ref33}\cite{ref34}, we think it is a trick, but it is not a good guarantee for whether the algorithm itself is effective or not. So we do not use the learning rate decay strategy when comparing PadamP with optimizers such as Adam, but used a constant learning rate. We then find that the experimental results were not very satisfactory. However, we also find some laws: when $ p=1/4 $ and $ p=1/5 $, the test accuracy and train loss of PadamP in the first half of the iteration (before $ 100 $ epochs) are very good, and the train loss in the second half of the iteration (after $ 100 $ epochs) is very average; however, in the second half of the iteration (after $ 100 $ epochs), the train loss of PadamP is very good at $ p=1/8 $, which inspires us to use the learning rate decay strategy to compare PadamP with Adam and other optimizers. The law inspires us that we can adjust the $ p $-value in an adaptive way, which has achieved better experimental results. We do the following experiment: at the $ 100th $ epoch, $ p $  reduced from $ 1/4 $ to $ 1/8 $, from $ 1/5 $ to $ 1/10 $, and from $ 1/8 $ to $ 1/16 $, with the initial learning rate constant.
	
	\begin{figure}
		\begin{minipage}[t]{0.5\linewidth}
			\centering
			\subfloat[Train Loss]{\includegraphics[width= 7cm,height=4.5cm]{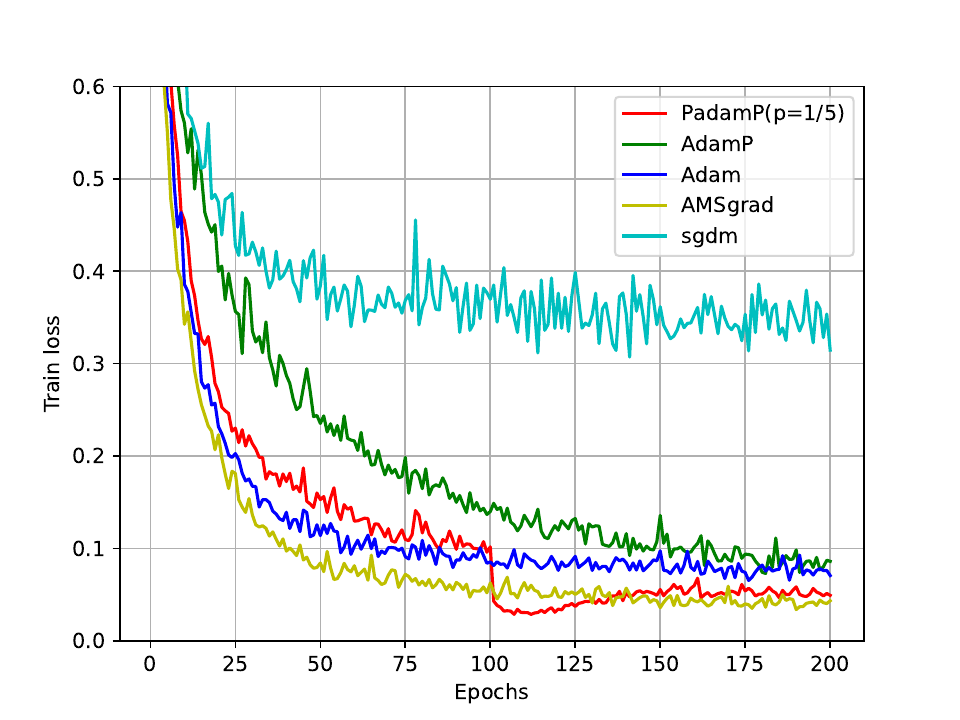}}\\
			\subfloat[Train Accuracy]{\includegraphics[width= 7cm,height=4.5cm]{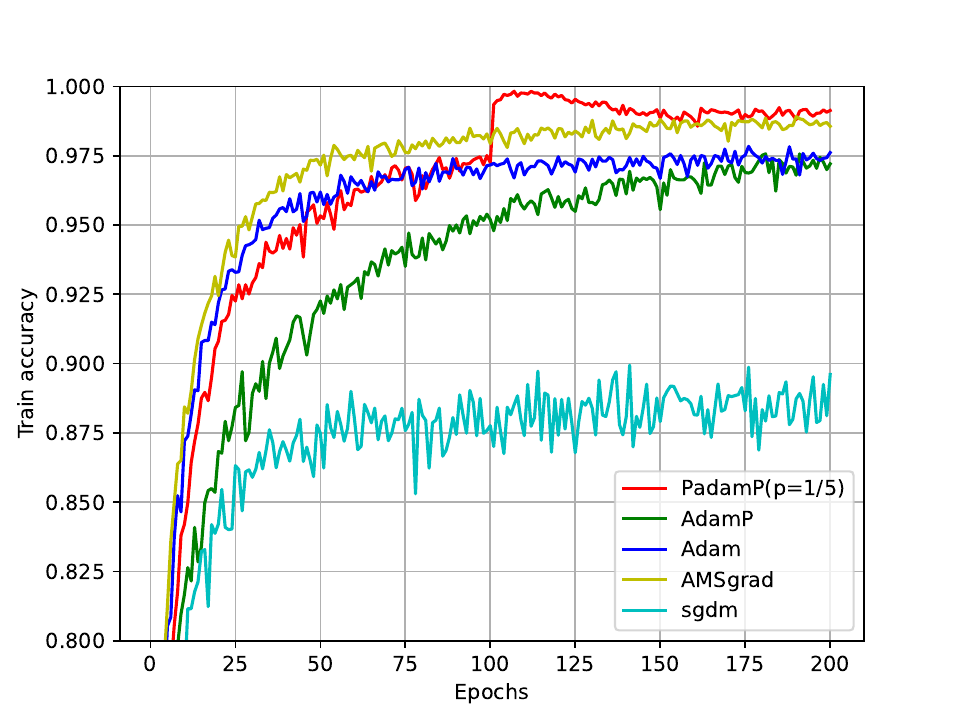}}\\
			\subfloat[Test Accuracy]{\includegraphics[width= 7cm,height=4.5cm]{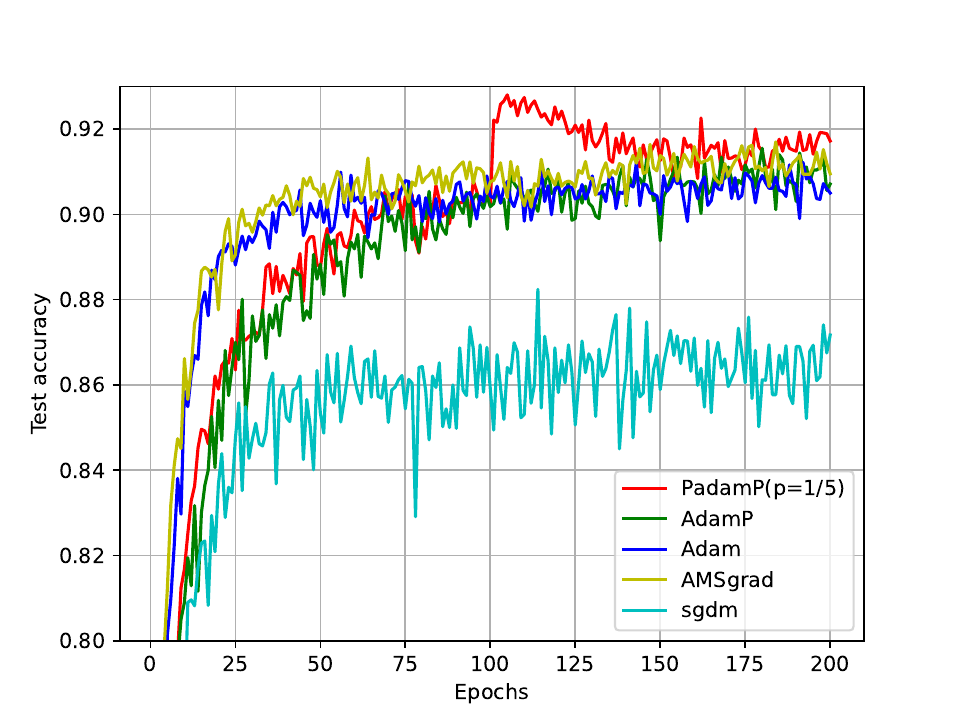}}
			\caption{PadamP V.S. other optimizers. Train on VGG-16 network, CIFAR-10 data(learning rate not decay).}
			\label{fig8}
		\end{minipage}%
		\hspace{0.5cm}
		\begin{minipage}[t]{0.5\linewidth}
			\centering
			\subfloat[Train Loss]{\includegraphics[width= 7cm,height=4.5cm]{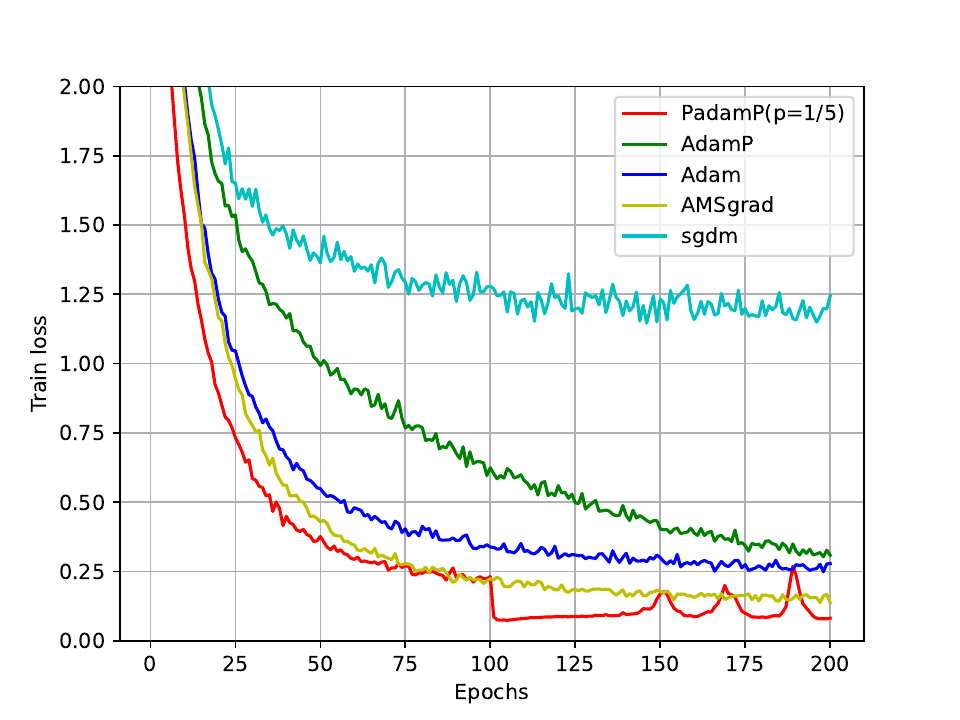}}\\
			\subfloat[Train Accuracy]{\includegraphics[width= 7cm,height=4.5cm]{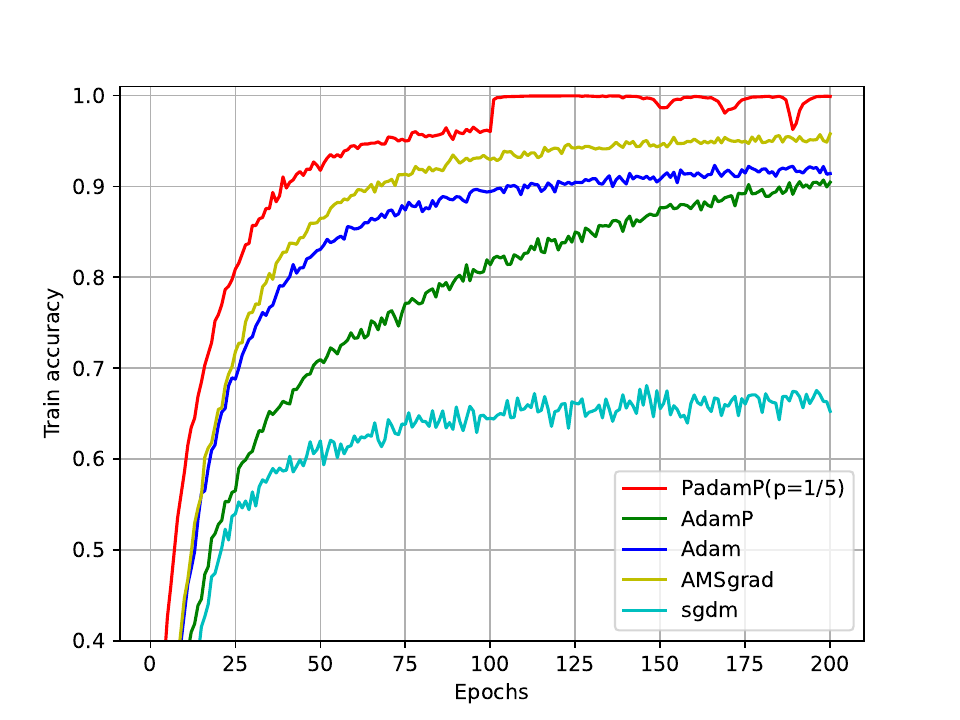}}\\
			\subfloat[Test Accuracy]{\includegraphics[width= 7cm,height=4.5cm]{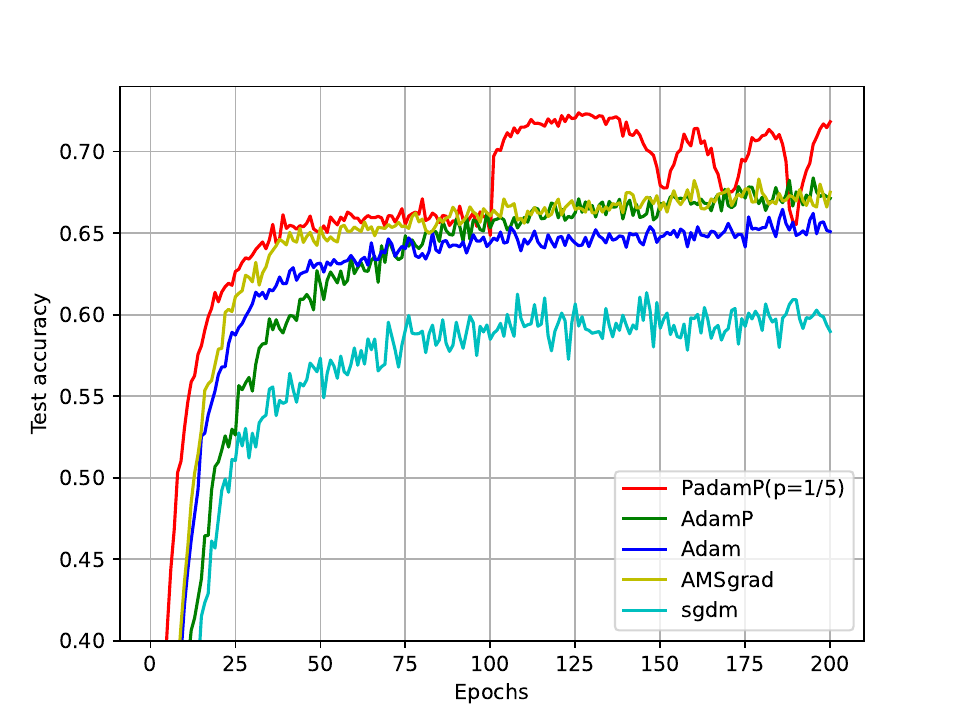}}
			\caption{PadamP V.S. other optimizers. Train on VGG-16 network, CIFAR-100 data(learning rate not decay).}
			\label{fig9}
		\end{minipage}
	\end{figure}
	
	\begin{figure}
		\begin{minipage}[t]{0.5\linewidth}
			\centering
			\subfloat[Train Loss]{\includegraphics[width= 7cm,height=4.5cm]{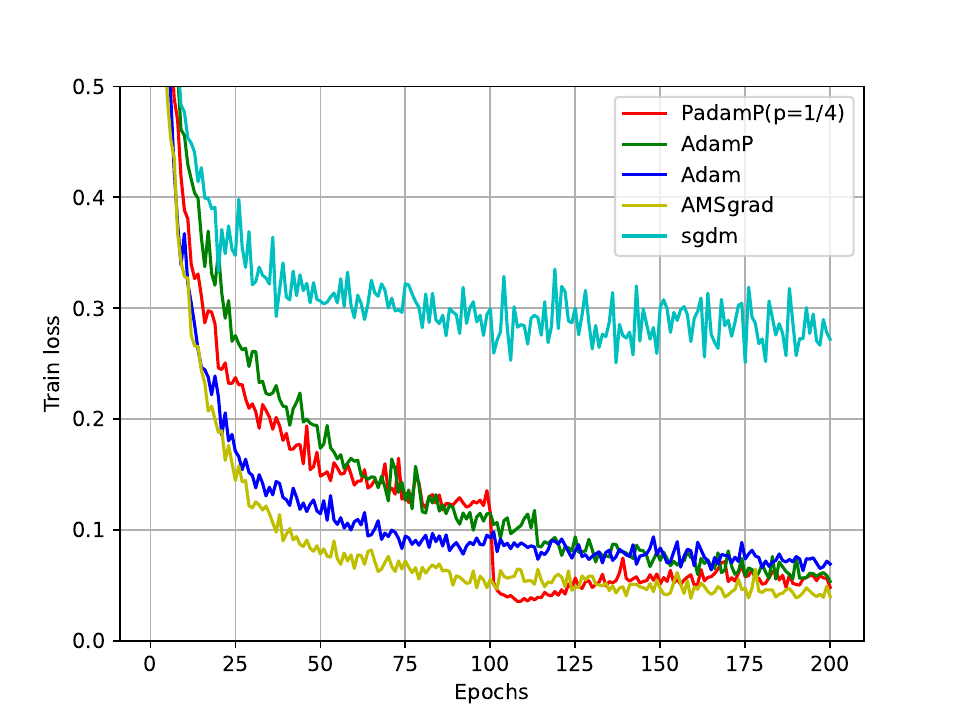}}\\
			\subfloat[Train Accuracy]{\includegraphics[width= 7cm,height=4.5cm]{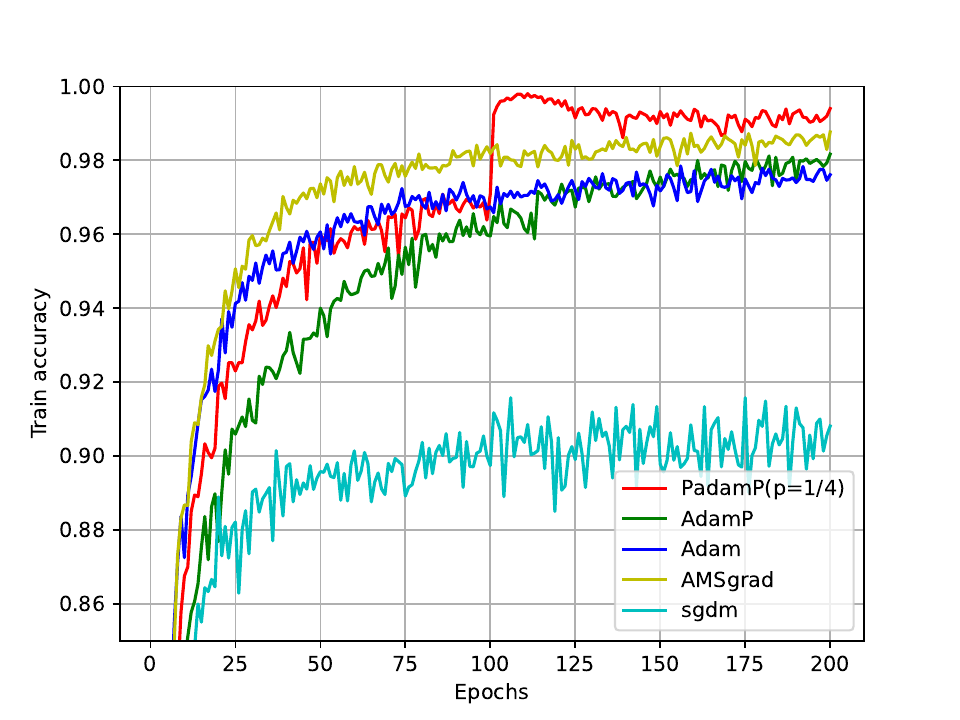}}\\
			\subfloat[Test Accuracy]{\includegraphics[width= 7cm,height=4.5cm]{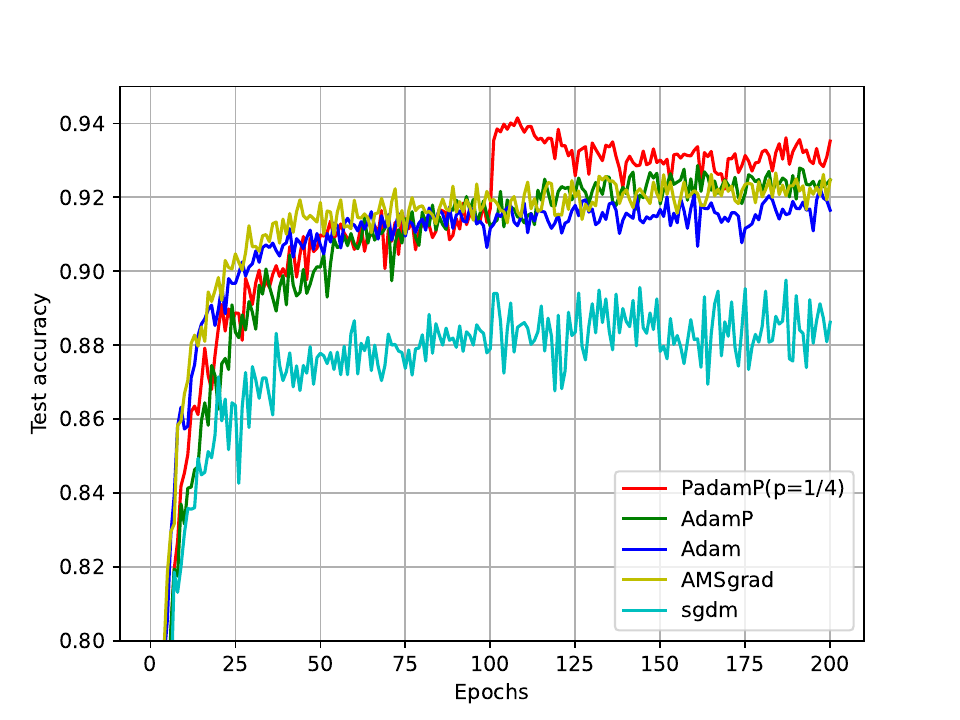}}
			\caption{PadamP V.S. other optimizers. Train on ResNet-18 network, CIFAR-10 data(learning rate not decay).}
			\label{fig10}
		\end{minipage}%
		\hspace{0.5cm}
		\begin{minipage}[t]{0.5\linewidth}
			\centering
			\subfloat[Train Loss]{\includegraphics[width= 7cm,height=4.5cm]{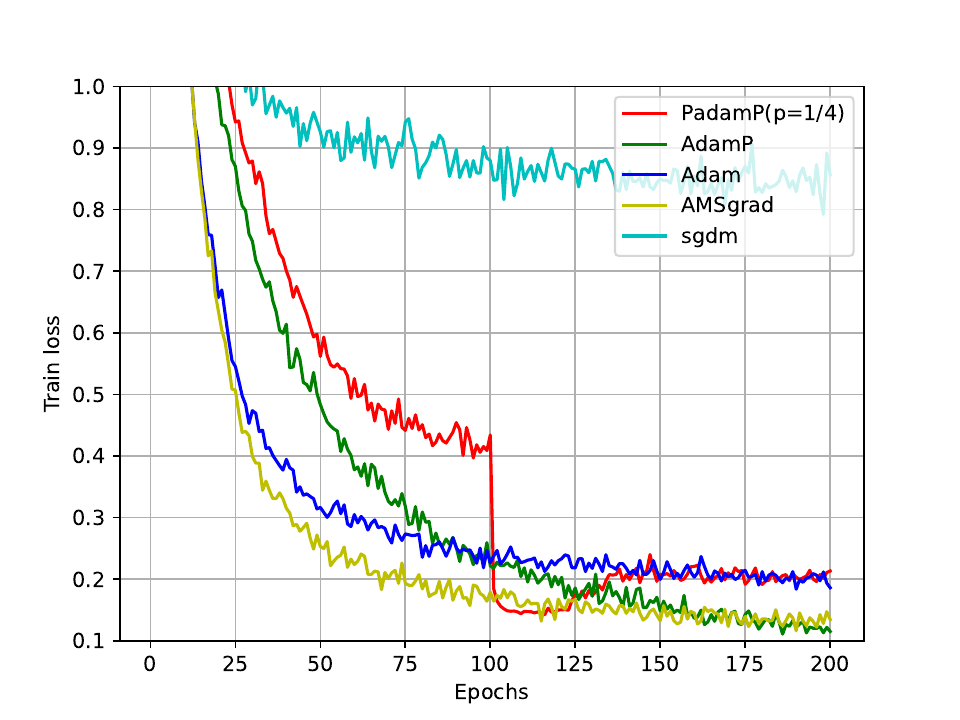}}\\
			\subfloat[Train Accuracy]{\includegraphics[width= 7cm,height=4.5cm]{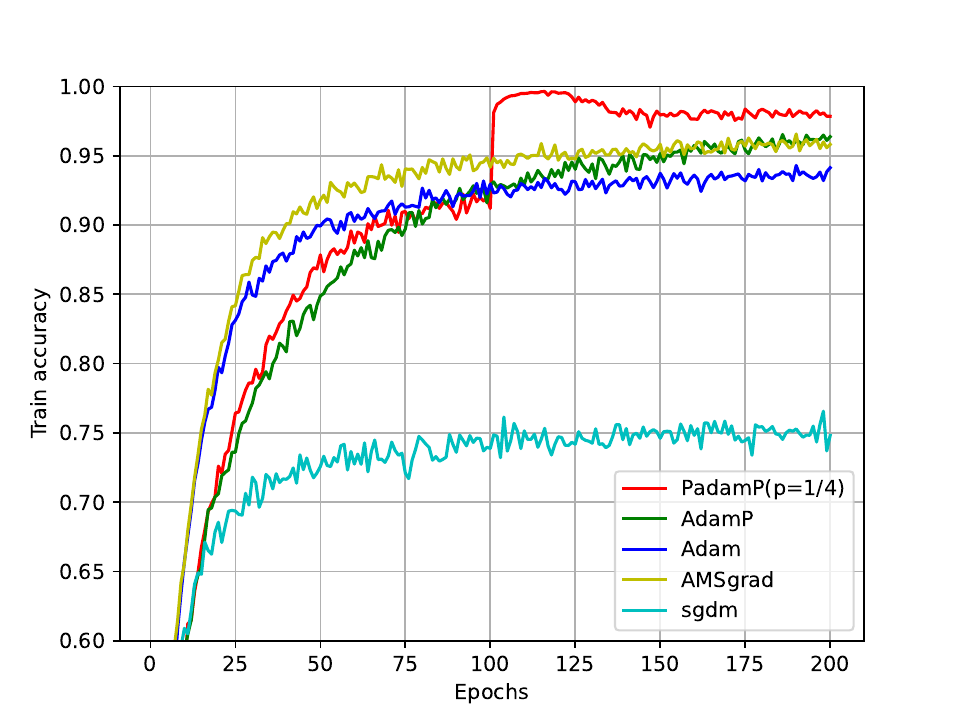}}\\
			\subfloat[Test Accuracy]{\includegraphics[width= 7cm,height=4.5cm]{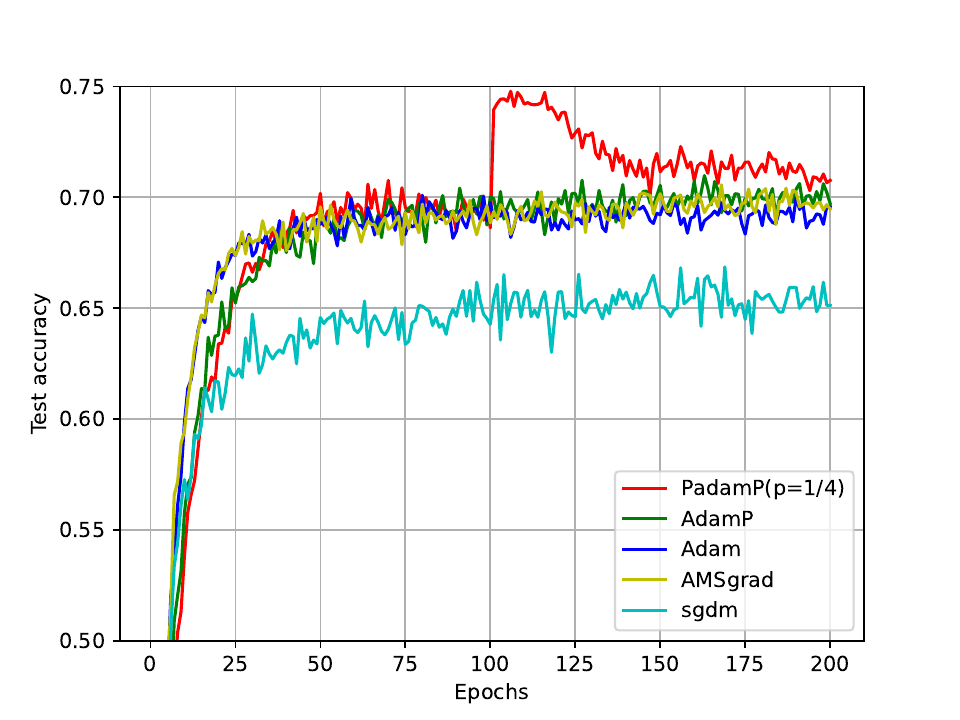}}
			\caption{PadamP V.S. other optimizers. Train on ResNet-18 network, CIFAR-100 data(learning rate not decay).}
			\label{fig11}
		\end{minipage}
	\end{figure}

	Figure \ref{fig8}-\ref{fig11} show the results of the experiments of 
	VGG-16 and ResNet-18 on CIFAR-10 and CIFAR-100, respectively.

	By looking at the figure \ref{fig8}-\ref{fig11}, we can find that PadamP performs well on both VGG-16 and ResNet-18 networks when the $ p $-value decreases, especially in the test accuracy, which is far better than the other optimizers, and also the train loss is not comparable with other methods, which shows the superiority and stability of the PadamP.
	
	\begin{figure}
		\begin{minipage}[t]{0.5\linewidth}
			\centering
			\subfloat[Train Loss]{\includegraphics[width= 7cm,height=4.5cm]{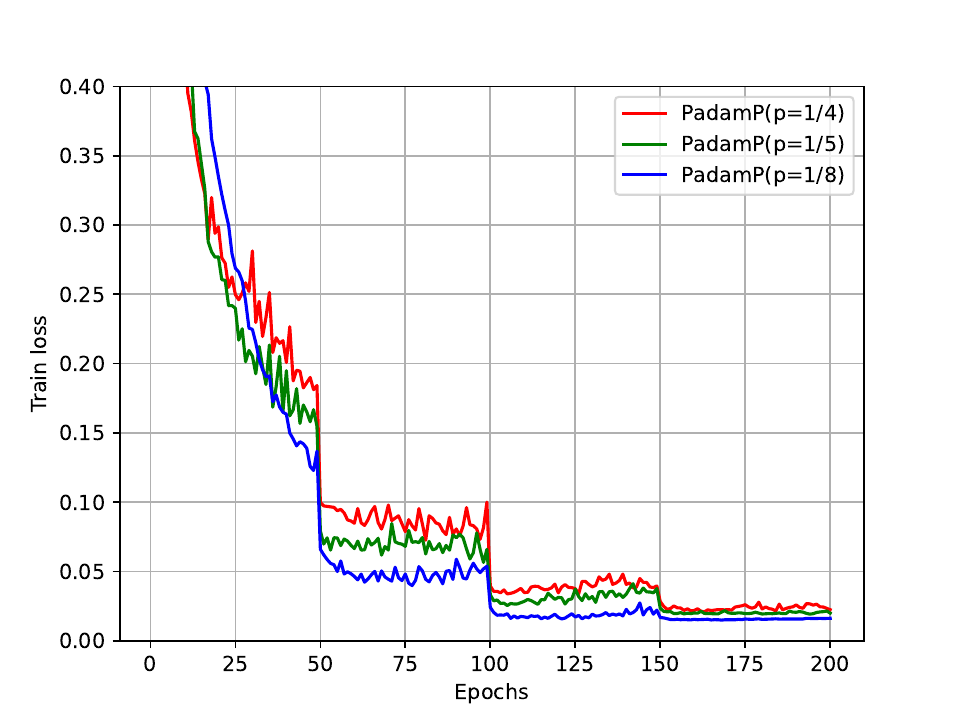}}\\
			\subfloat[Train Accuracy]{\includegraphics[width= 7cm,height=4.5cm]{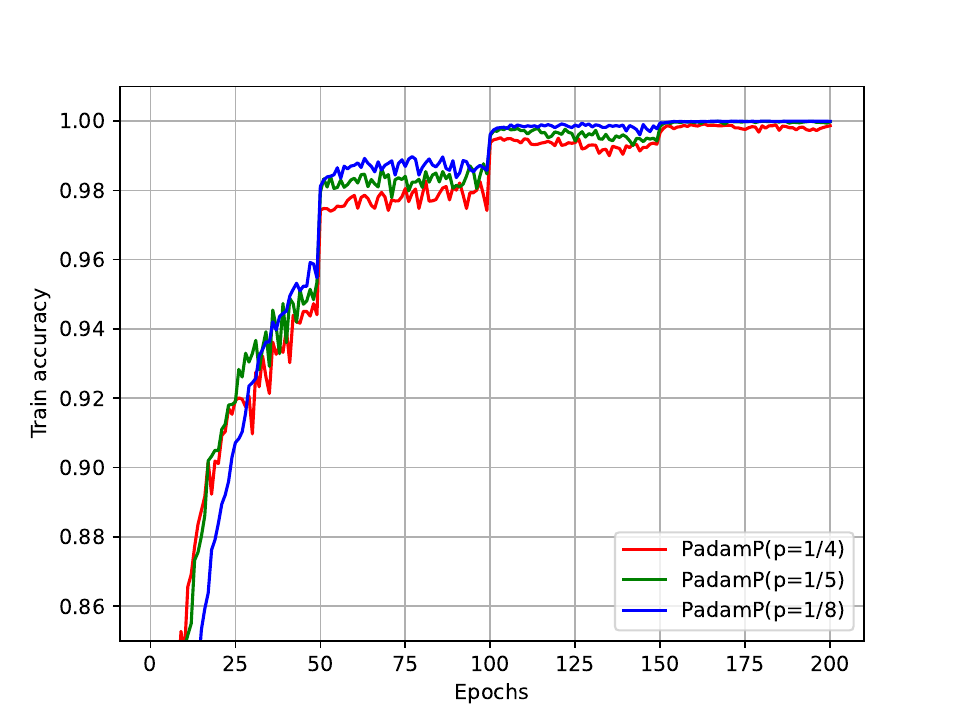}}\\
			\subfloat[Test Accuracy]{\includegraphics[width= 7cm,height=4.5cm]{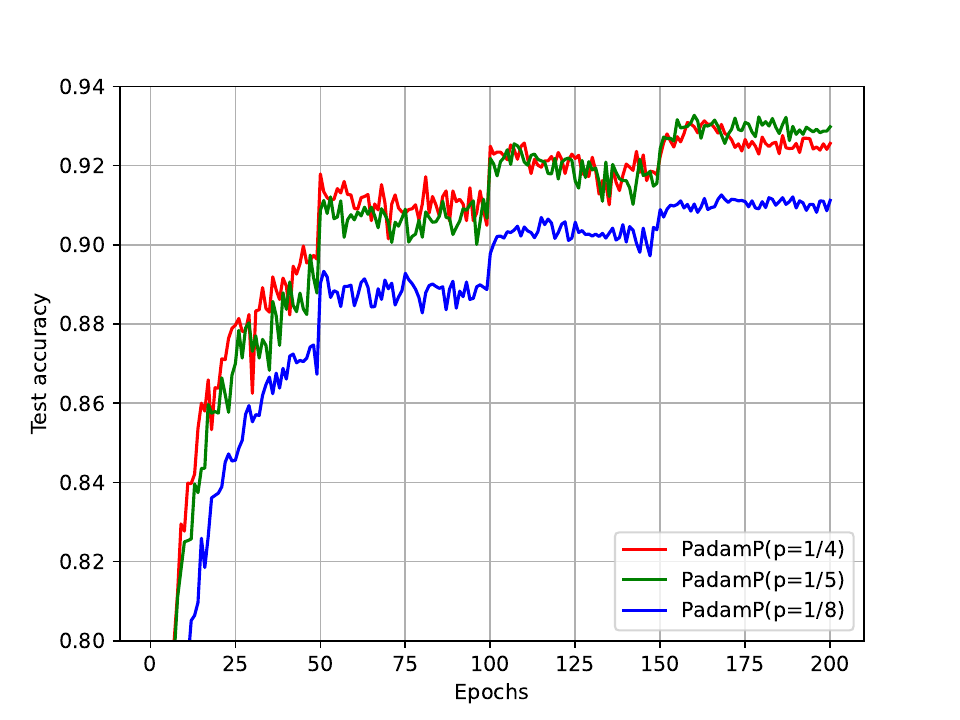}}
			\caption{Different p-values trained on VGG-16 network, CIFAR-10 data(learning rate decay).}
			\label{fig12}
		\end{minipage}%
		\hspace{0.5cm}
		\begin{minipage}[t]{0.5\linewidth}
			\centering
			\subfloat[Train Loss]{\includegraphics[width= 7cm,height=4.5cm]{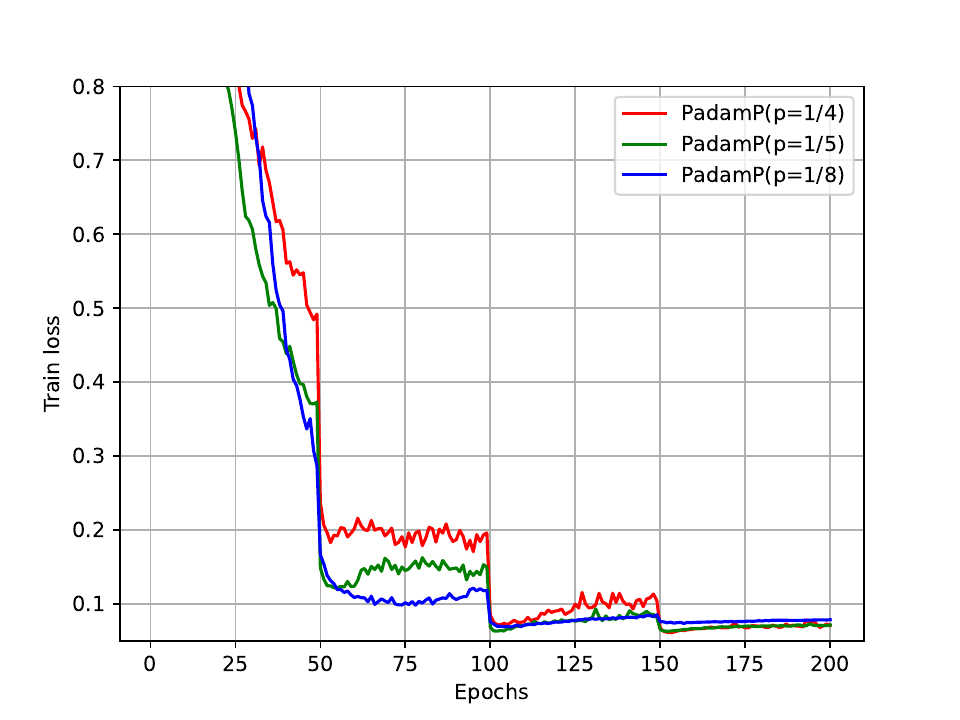}}\\
			\subfloat[Train Accuracy]{\includegraphics[width= 7cm,height=4.5cm]{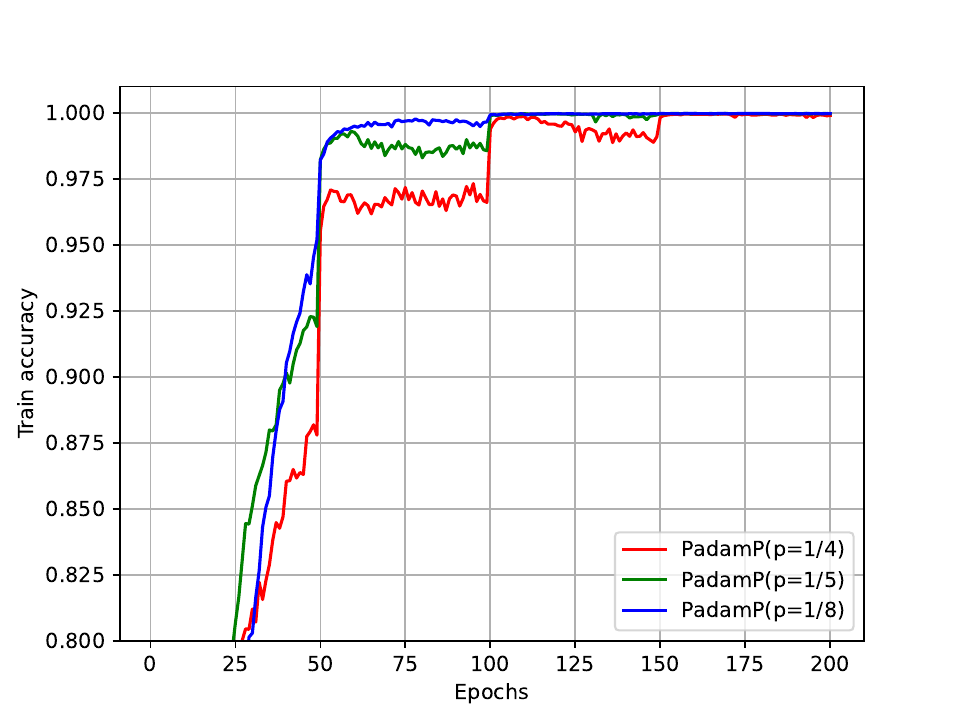}}\\
			\subfloat[Test Accuracy]{\includegraphics[width= 7cm,height=4.5cm]{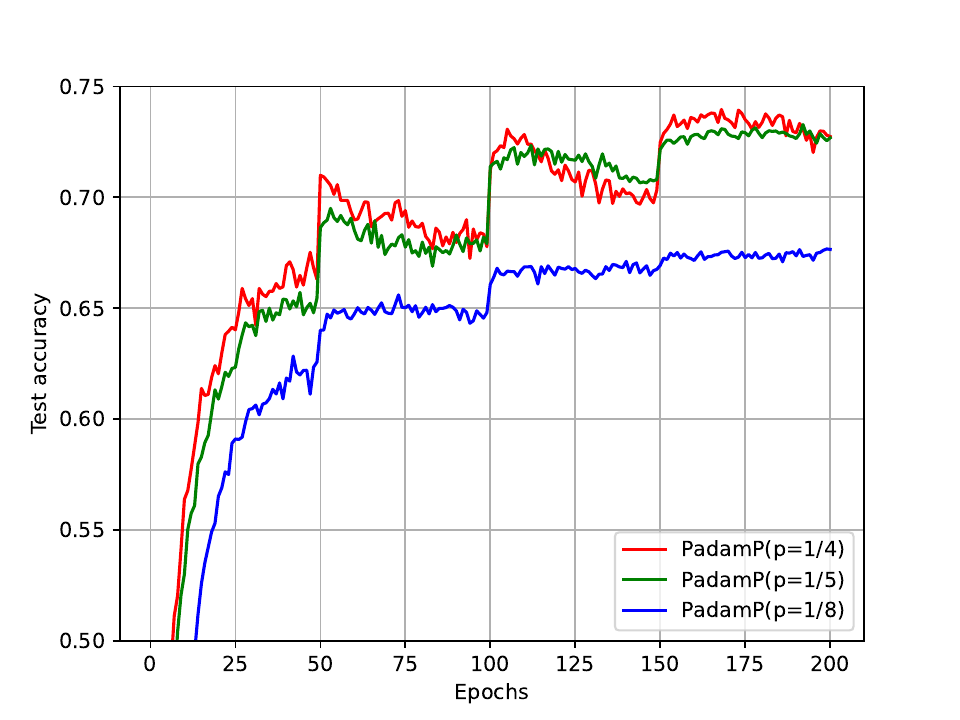}}
			\caption{Different p-values trained on VGG-16 network, CIFAR-100 data (learning rate decay).}
			\label{fig13}
		\end{minipage}
	\end{figure}

	\begin{figure}
		\begin{minipage}[t]{0.5\linewidth}
			\centering
			\subfloat[Train Loss]{\includegraphics[width= 7cm,height=4.5cm]{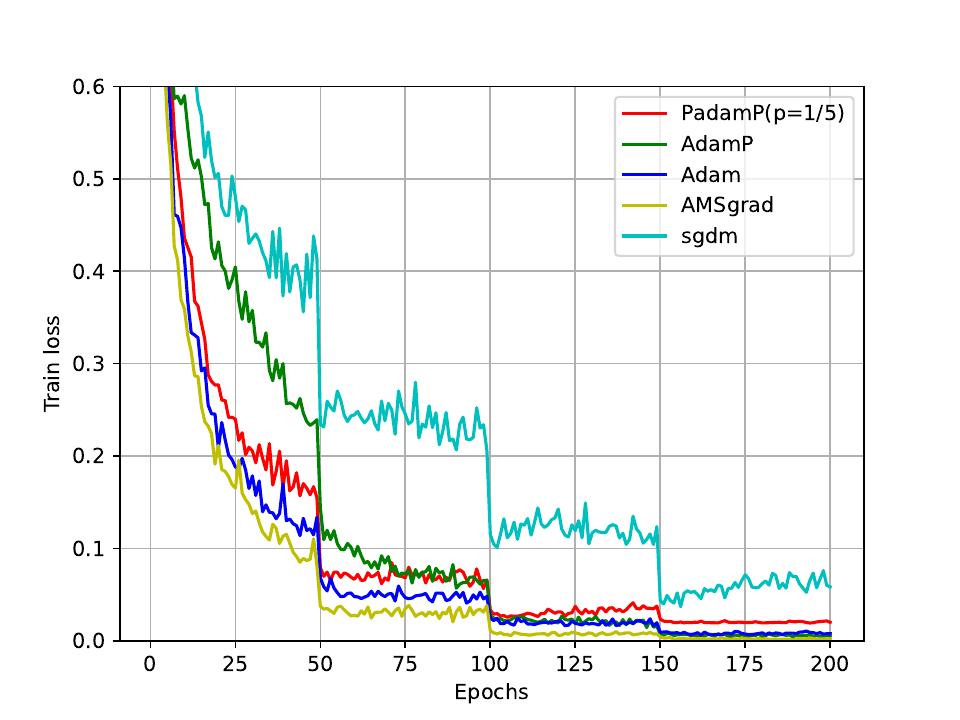}}\\
			\subfloat[Train Accuracy]{\includegraphics[width= 7cm,height=4.5cm]{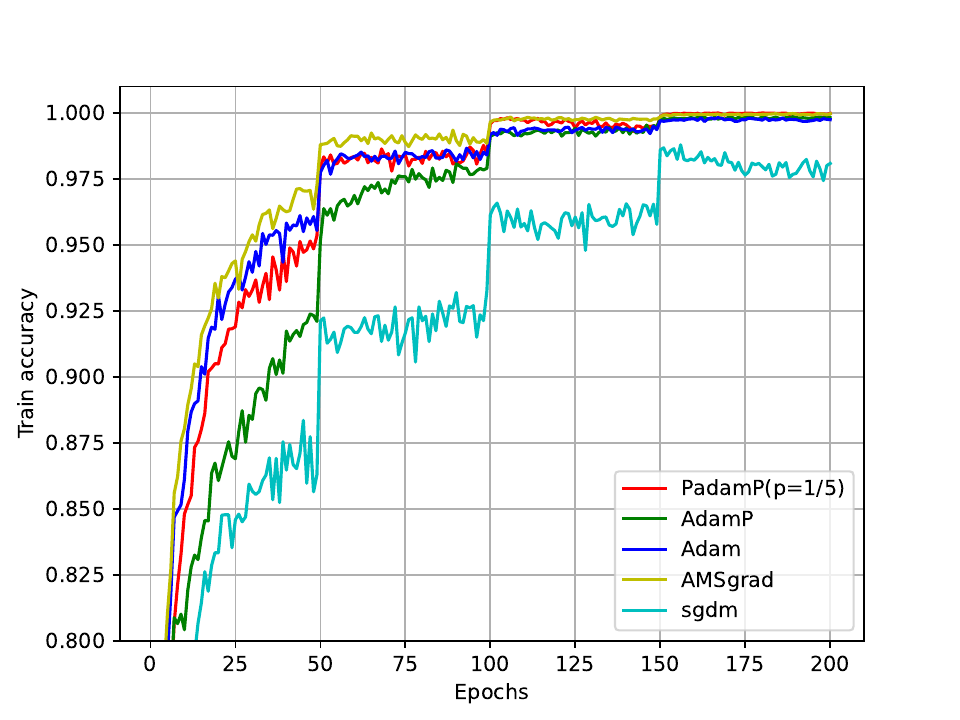}}\\
			\subfloat[Test Accuracy]{\includegraphics[width= 7cm,height=4.5cm]{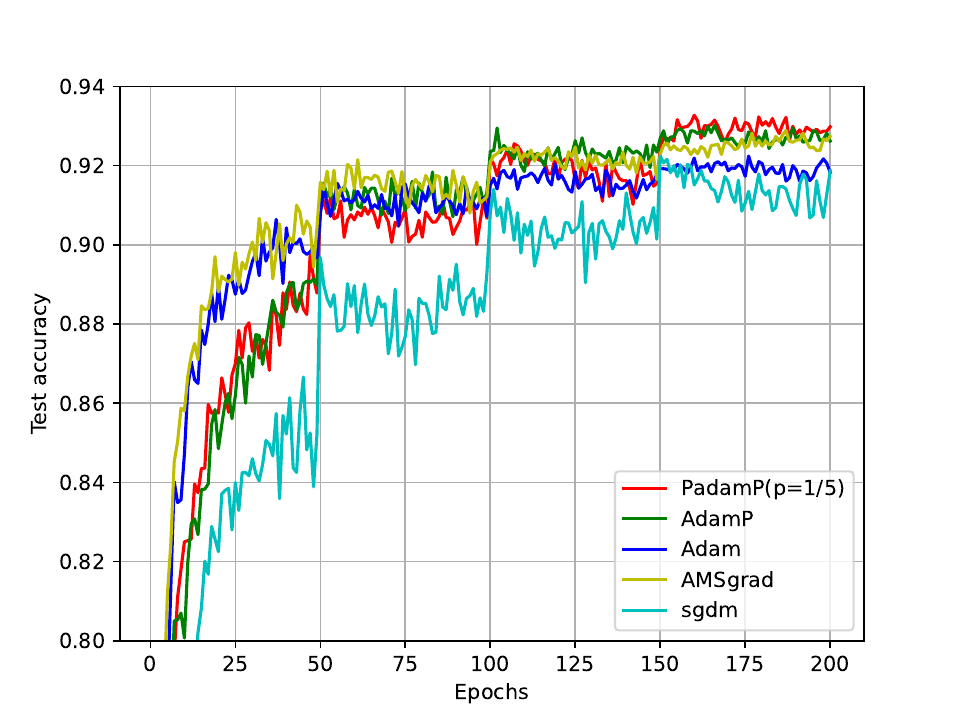}}
			\caption{PadamP V.S. other optimizers. Train on VGG-16 network, CIFAR-10 data(learning rate decay).}
			\label{fig14}
		\end{minipage}%
		\hspace{0.5cm}
		\begin{minipage}[t]{0.5\linewidth}
			\centering
			\subfloat[Train Loss]{\includegraphics[width= 7cm,height=4.5cm]{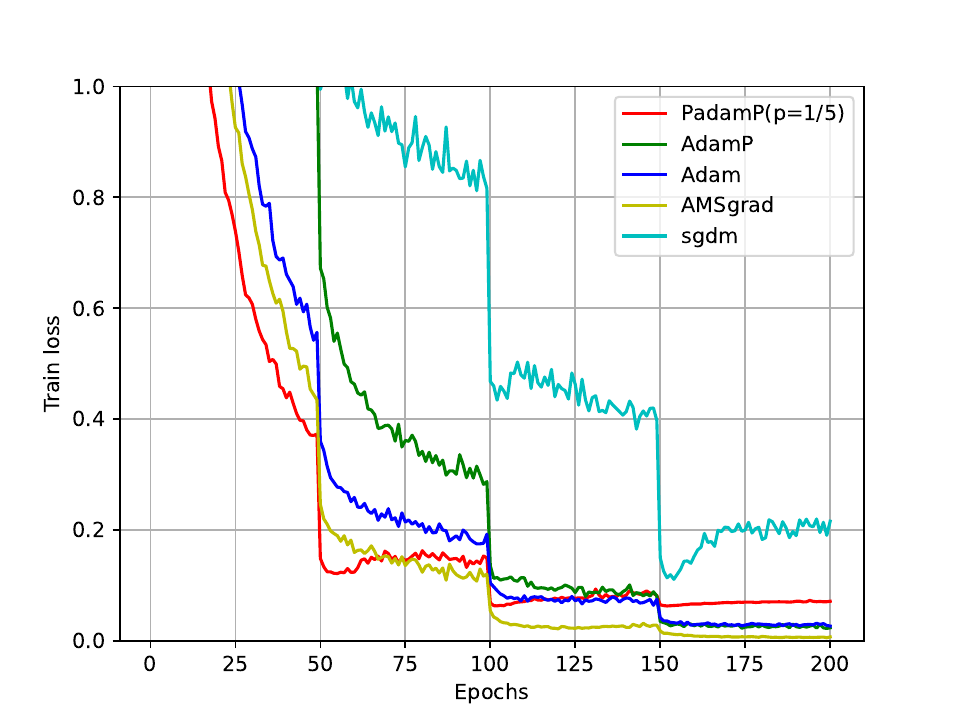}}\\
			\subfloat[Train Accuracy]{\includegraphics[width= 7cm,height=4.5cm]{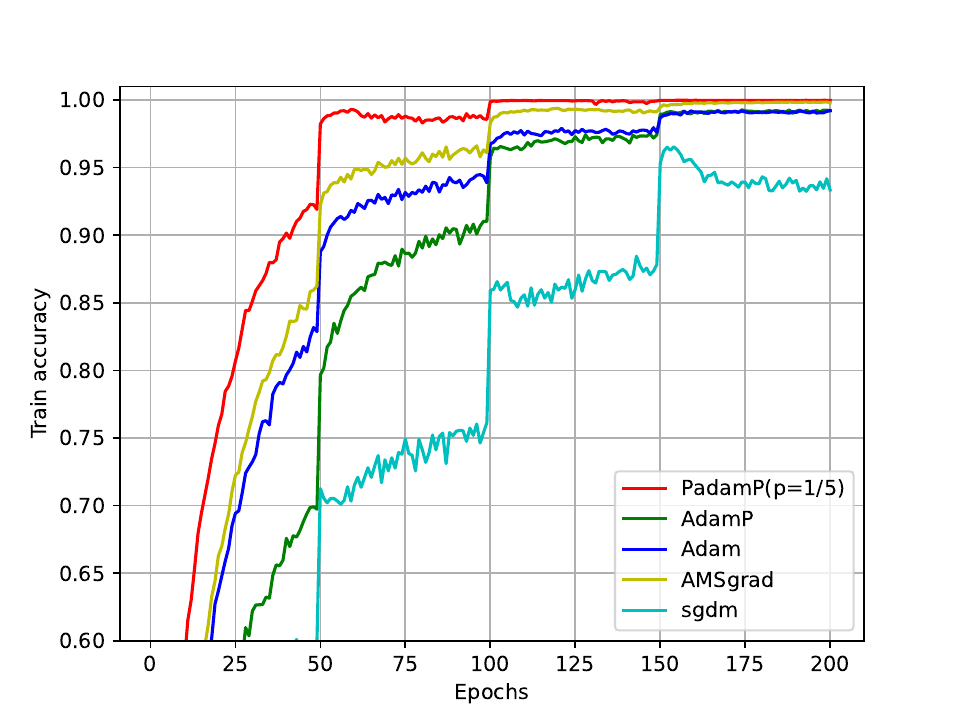}}\\
			\subfloat[Test Accuracy]{\includegraphics[width= 7cm,height=4.5cm]{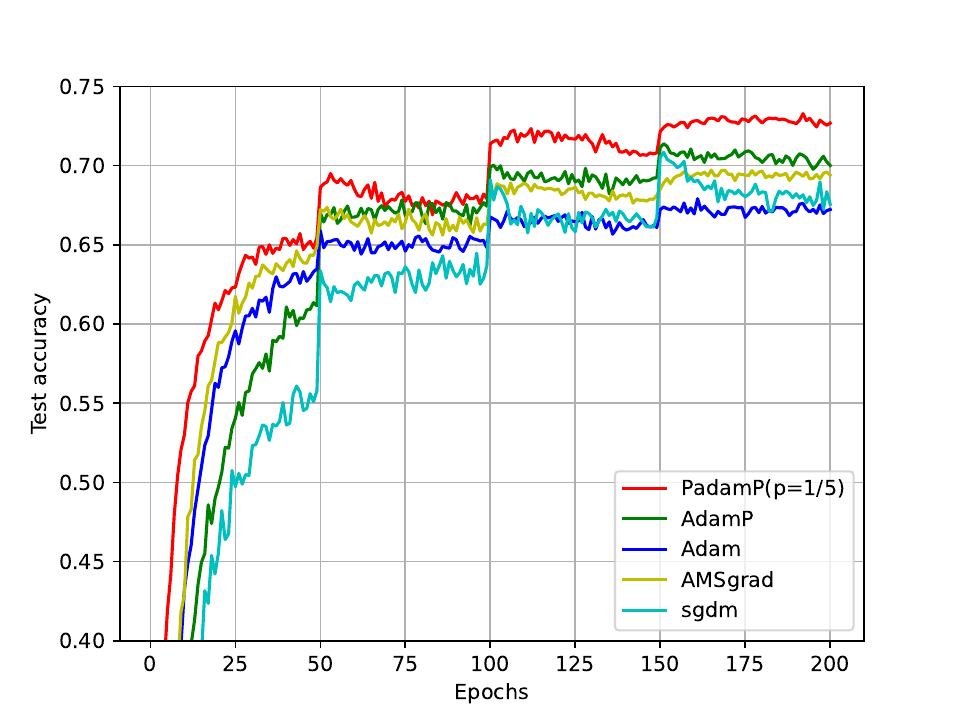}}
			\caption{PadamP V.S. other optimizers. Train on VGG-16 network, CIFAR-100 data(learning rate decay).}
			\label{fig15}
		\end{minipage}
	\end{figure}
	
	Figure \ref{fig12}-\ref{fig15} show the results of the experiments of VGGNet-16 on CIFAR-10 and CIFAR-100.
	Here we use the initial learning rate decay strategy with 0.1 decay per 50 epochs, and the experimental results corroborate what we said above:
	When $ p=1/4 $ and $ p=1/5 $, the test accuracy and train loss of PadamP in the first half of the iteration (before $ 100 $ epochs) are very good; When $ p=1/8 $ in the second half of the iteration (after $ 100 $ epochs), the train loss of PadamP performs well.
	
	\section{Conclusion}
	\label{}
	In this paper, for the purpose of accelerating deep neural networks training and 
	helping the network find more optimal parameters, the projection-gradient is 
	incorporated into the generic Adam, which is named PadamP. 
	The PadamP is modified from vanilla AdamP, gradient accumulation adopts $ p $-th power adaptive strategy, while updating step size adopts projection-gradient. 
	We theoretically prove the convergence
	of our algorithm and manage to not only provide convergence for the non-convex but 
	also deal with the coupling of the first-order moment estimation coefficients $\beta_{1}$ and the second-order moment estimation coefficients $\beta_{2}$.
	Numerical experiments of training VGG-16/ResNet-18 on CIFAR-10/100 for image classification 
	demonstrate effectiveness and better performance. Our algorithm performs more stable and 
	arrives at 100\% train accuracy faster than other optimizers. Higher test accuracy provides strong evidence 
	that our algorithm obtains more optimal parameters of deep neural networks.
	More future work includes conducting various experiments for deep learning task, 
	applying variance reduction technique for the PadamP, exploring 
	line-search for finding a suitable stepsize, etc. 
	
	\bibliographystyle{unsrt}
	\bibliography{references}
	
	\appendices
	\section{Proof of Some Lemmas} 
	\label{prth31}
	Before beginning the proof of Theorem \ref{th3.1}, all the necessary lemmas as well as the proof of them are firstly provided at once.
	
	\begin{lemma}\label{lem2}
		Given that $ m_{t} $ in PadamP, the following equation holds:
		\begin{equation}\label{eqA42}
			-m_{t}=-g_{t}+\frac{\beta_{1t}}{1-\beta_{1t}}\left(m_{t}-m_{t-1}\right).
		\end{equation}
	\end{lemma}

	\begin{proof}
		By referring to the definition, the equation $ -g_{t}+\frac{\beta_{1t}}{1-\beta_{1t}}\left(m_{t}-m_{t-1}\right) $ can be rewritten as follows:
		\begin{equation}\label{eqA44}
			-g_{t}+\frac{\beta_{1t}}{1-\beta_{1t}}\left(m_{t}-m_{t-1}\right)=\frac{-\left(1-\beta_{1t}\right) g_{t}+\beta_{1t} m_{t}-\beta_{1t} m_{t-1}}{1-\beta_{1t}},
		\end{equation}
		When the equation  $ m_{t}=\beta_{1t} m_{t-1}+\left(1-\beta_{1t}\right) g_{t} $ is substituted into the above equation, we obtain:
		\begin{equation}\label{eqA45}
			-g_{t}+\frac{\beta_{1t}}{1-\beta_{1t}}\left(m_{t}-m_{t-1}\right)=\frac{\beta_{1t} m_{t}-m_{t}}{1-\beta_{1t}}=-m_{t}.
		\end{equation}
		
		The proof is over.
	\end{proof}

	\begin{lemma}\label{lem3}
		For $ v_{t} $ defined as in PadamP, under the condition that the gradient satisfies Assumption \ref{ass31}, the following inequality is valid:
		\begin{equation}\label{eqA42} 
			0 \le {v_t} \le C_1^2.
		\end{equation}
	\end{lemma}
	
	\begin{proof}
		According to PadamP, $ v_{t} $ can be expressed as:
		\begin{equation}\label{eqA46}
			v_{t}=\beta_{2} v_{t-1}+(1-\beta_{2})\left\|g_{t}\right\|^{2}=(1-\beta_{2}) \sum_{i=1}^{t} \beta_{2}^{t-i}\left\|g_{t}\right\|^{2},
		\end{equation}
		From the definition of $ v_{t} $, it is easy to see that $ v_{t} \ge 0 $ holds. additionally, by Assumption \ref{ass31} and considering that $\beta_{2} \in (0,1)$ , we can derive the following inequality:
		\begin{equation}\label{eqA47}
			v_{t}=(1-\beta_{2}) \sum_{i=1}^{t} \beta_{2}^{t-i}\left\|g_{t}\right\|^{2} \leq(1-\beta_{2}) C_{1}^{2} \sum_{i=1}^{t} \beta_{2}^{t-i}=\left(1-\beta_{2}^{t}\right) C_{1}^{2} \leq C_{1}^{2}.
		\end{equation}
		
		The proof is over.
		
	\end{proof}

	\begin{lemma}\label{lem4}
		Consider a diagonal matrix $ B_{t} $ that satisfies the following form:
		\begin{equation}\label{eq18}
			B_{t}=\left[\begin{array}{ccc}
				\frac{1}{(v_{t, 1}+\epsilon)^{p}} & \cdots & 0 \\
				\vdots & \ddots & \vdots \\
				0 & \cdots & \frac{1}{(v_{t, d}+\epsilon)^{p}}
			\end{array}\right], 	
		\end{equation}
		given that the initial value of $ v_{t} $ is set to $ 0 $ and Assumption \ref{ass31} holds, the following inequality can be established:
		\begin{equation}\label{eq19}
			\frac{I}{({C_{1}^{2}+\epsilon})^{p}} \leq B_{t} \leq \frac{I}{{\epsilon}^{p}},
		\end{equation}
		where $ I $ is the unit matrix, $ \epsilon $ is a constant and $ \epsilon \in(0,1) $.
	\end{lemma}
	
	\begin{proof}
		The diagonal matrix $ B_{t} $ is defined as
		\begin{equation}\label{eq20}
			B_{t}=\left[\begin{array}{ccc}
				\frac{1}{({v_{t, 1}+\epsilon})^{p}} & \cdots & 0 \\
				\vdots & \ddots & \vdots \\
				0 & \cdots & \frac{1}{({v_{t, d}+\epsilon})^{p}}
			\end{array}\right],
		\end{equation}
		according to PadamP, the elements of $ v_{t} $ can be expressed as
		\begin{equation}\label{eq21}
			v_{t, i}=(1-\beta_{2}) \sum_{j=1}^{t} \beta_{2}^{t-i} g_{j-1,i}^{2},
		\end{equation}
		it can be readily observed that each element is  ${v_{t,i}}$ non-negative and ${v_{t,i}} \ge 0$. By Assumption \ref{ass31} and considering that $\beta_{2} \in (0,1)$, it follows that
		\begin{equation}\label{eq22}
			v_{t, i} \leq C_{1}^{2},
		\end{equation}
		based on the above results, for the diagonal matrix $ B_{t} $ has the conclusion
		\begin{equation}\label{eq23}
			\frac{1}{({C_{1}^{2}+\epsilon})^{p}} I=\left[\begin{array}{ccc}
				\frac{1}{({C_{1}^{2}+\epsilon})^{p}} & \cdots & 0 \\
				\vdots & \ddots & \vdots \\
				0 & \cdots & \frac{1}{({C_{1}^{2}+\epsilon})^{p}}
			\end{array}\right] \leq B_{t}=\left[\begin{array}{ccc}
				\frac{1}{({\sum_{j=1}^{t} g_{j-1,1}^{2}+\epsilon})^{p}} & \cdots & 0 \\
				\vdots & \ddots & \vdots \\
				0 & \cdots & \frac{1}{({\sum_{j=1}^{t} g_{j-1,d}^{2}+\epsilon})^{p}}
			\end{array}\right]
		\end{equation}
		and
		\begin{equation}\label{eq24}
			B_{t}=\left[\begin{array}{ccc}
				\frac{1}{({\sum_{j=1}^{t} g_{j-1,1}^{2}+\epsilon})^{p}} & \cdots & 0 \\
				\vdots & \ddots & \vdots \\
				0 & \cdots & \frac{1}{({\sum_{j=1}^{t} g_{j-1,d}^{2}+\epsilon})^{p}}
			\end{array}\right] \leq\left[\begin{array}{ccc}
				\frac{1}{{\epsilon}^{p}} & \cdots & 0 \\
				\vdots & \ddots & \vdots \\
				0 & \cdots & \frac{1}{{\epsilon}^{p}}
			\end{array}\right]=\frac{1}{{\epsilon}^{p}} I.
		\end{equation}
		
		The proof is over.

	\end{proof}
	\begin{lemma}\label{lem5}
		Suppose that $ g_{t} $, $ v_{t} $ satisfy Assumption \ref{ass31} and Assumption  \ref{ass35} respectively, and both Lemma \ref{lem3} and Lemma \ref{lem4} hold. Then, the following inequalities can be established:
		\begin{equation}\label{eq19}
			\left\langle\hat{\theta}_{t} , \frac{m_{t}}{\left(v_{t}+\epsilon\right)^{p}}\right\rangle  \leq \frac{C_{1}}{\epsilon^{p}},
		\end{equation}
		\begin{equation}\label{eq19}
			\mathbb{E}\left[\left\|\frac{g_{t}}{({v_{t}+\epsilon})^{p}}\right\|^{2}\right] \leq \frac{C_{1}^{2}}{\epsilon^{2p}},
		\end{equation}
		\begin{equation}\label{eq19}
			\mathbb{E}\left[\left\langle\nabla f\left(\theta_{t}\right),  \frac{m_{t}-m_{t-1}}{({v_{t}+\epsilon})^{p}}\right\rangle\right] \leq \frac{2C_{1}^{2}}{\epsilon^{p}}.
		\end{equation}
	\end{lemma}
	\begin{proof}
		Given that $\| \hat{\theta}_{t} \| = 1 $, and leveraging Assumption \ref{ass31}, Assumption  \ref{ass35}, Lemma \ref{lem3} and Lemma \ref{lem4},  the following chain of inequalities is derived.
		First, by the property of the dot-product and the norm, we know that $ 		\left\langle\hat{\theta}_{t} , \frac{m_{t}}{\left(v_{t}+\epsilon\right)^{p}}\right\rangle $ is bounded as follows:
		\begin {equation}\label {eq19}
		\left\langle\hat{\theta}_{t} , \frac{m_{t}}{\left(v_{t}+\epsilon\right)^{p}}\right\rangle\leq \left\|\hat{\theta}_{t}\right\|\left\|\frac{m_{t}}{\left(v_{t}+\epsilon\right)^{p}}\right\|.
		\end {equation}
		Since  $ \left\|\hat{\theta}_{t}\right\|=1  $, we further have  $ \left\|\hat{\theta}_{t}\right\|\left\|\frac{m_{t}}{\left(v_{t}+\epsilon\right)^{p}}\right\|=\left\|\frac{m_{t}}{\left(v_{t}+\epsilon\right)^{p}}\right\| $. According to the norm properties and the given assumptions,
		\begin{equation}
			\left\|\frac{m_{t}}{\left(v_{t}+\epsilon\right)^{p}}\right\| \leq\left\|m_{t}\right\|\left\|\frac{1}{\left(v_{t}+\epsilon\right)^{p}}\right\|  .
		\end{equation}
		From Lemma \ref{lem3} and Lemma \ref{lem4},  $\left\| \frac{1}{\left(v_{t}+\epsilon\right)^{p}} \right\| \leq \frac{1}{\epsilon^{p}}$, thus
		\begin{equation}
			\left\langle\hat{\theta}_{t} , \frac{m_{t}}{\left(v_{t}+\epsilon\right)^{p}}\right\rangle\leq \left\|\hat{\theta}_{t}\right\|\left\|\frac{m_{t}}{\left(v_{t}+\epsilon\right)^{p}}\right\| 
			\leq \left\|m_{t}\right\|\left\|\frac{1}{\left(v_{t}+\epsilon\right)^{p}}\right\| 
			\leq \frac{C_{1}}{\epsilon^{p}}.
		\end{equation}
		
		Next, consider the norm $ \left\|\frac {g_{t}}{\left (v_{t}+\epsilon\right)^{p}}\right\| $. By similar reasoning as above, we have
		\begin {equation}\label {eq19}
		\left\|\frac {g_{t}}{\left (v_{t}+\epsilon\right)^{p}}\right\|
		\leq 
		\left\| g_{t} \right\| \left\| \frac {1}{\left (v_{t}+\epsilon\right)^{p}} \right\|.
		\end {equation}
		
		From Assumption \ref{ass31}, $\left\| g_{t} \right\| \leq C_{1}$ and $\left\| \frac{1}{\left(v_{t}+\epsilon\right)^{p}} \right\| \leq \frac{1}{\epsilon^{p}}$$\left\| \frac{1}{\left(v_{t}+\epsilon\right)^{p}} \right\| \leq \frac{1}{\epsilon^{p}}$. So  $\left\|\frac{g_{t}}{\left(v_{t}+\epsilon\right)^{p}}\right\| \leq \frac{C_{1}}{\epsilon^{p}}$. Taking the expectation of the square of the norm, we get
		\begin {equation}\label {eq19}
		\mathbb {E}\left [\left\|\frac {g_{t}}{\left (v_{t}+\epsilon\right)^{p}}\right\|^{2}\right] \leq \frac {C_{1}^{2}}{\epsilon^{2p}}.
		\end {equation}
		
		Finally, for the expectation $\mathbb{E}\left[\left\langle\nabla f\left(\theta_{t}\right),  \frac{m_{t}-m_{t-1}}{\left(v_{t}+\epsilon\right)^{p}}\right\rangle\right]$, by the Cauchy-Schwarz inequality, we have
		\begin {equation}\label {eq19}
		\mathbb {E}\left [\left\langle\nabla f\left (\theta_{t}\right), \frac {m_{t}-m_{t-1}}{\left (v_{t}+\epsilon\right)^{p}}\right\rangle\right] \leq 
		\left\|\nabla f\left (\theta_{t}\right)\right\| \left\|\frac {m_{t}-m_{t-1}}{\left (v_{t}+\epsilon\right)^{p}}\right\|.
		\end {equation}
		
		Using the triangle inequality for the norm, $\left\|\frac{m_{t}-m_{t-1}}{\left(v_{t}+\epsilon\right)^{p}}\right\| \leq (\left\|m_{t}\right\| + \left\|m_{t-1}\right\|)\left\|\frac{1}{\left(v_{t}+\epsilon\right)^{p}}\right\|$, from Assumption \ref{ass31}, $\|\nabla f\left(\theta_{t}\right)\| \leq C_{1}$, 
		$\left\|m_{t}\right\| \leq C_{1}$,$\left\|m_{t-1}\right\| \leq C_{1}$, and from Lemma \ref{lem3} and Lemma \ref{lem4}, $\left\|\frac{1}{\left(v_{t}+\epsilon\right)^{p}}\right\| \leq \frac{1}{\epsilon^{p}}$.
		
		Therefore,
		\begin {equation}\label {eq19}
		\begin {aligned}
		\mathbb {E}\left [\left\langle\nabla f\left (\theta_{t}\right), \frac {m_{t}-m_{t-1}}{\left (v_{t}+\epsilon\right)^{p}}\right\rangle\right]
		& \leq
		\left\|\nabla f\left (\theta_{t}\right)\right\| \left\|\frac {m_{t}-m_{t-1}}{\left (v_{t}+\epsilon\right)^{p}}\right\| \\
		& \leq
		\|\nabla f\left(\theta_{t}\right)\| (\left\|m_{t}\right\| + \left\|m_{t-1}\right\|)\left\|\frac{1}{\left(v_{t}+\epsilon\right)^{p}}\right\|
		\leq \frac {2C_{1}^{2}}{\epsilon^{p}}.
		\end {aligned}
		\end {equation}
		
		The proof is over.
	\end{proof}
	
	\section{Proof of Theorem 3.1} 
	\label{prth32}
	Given that $ f $ is $ L $-smooth, it can be inferred from the property of $ L $ - smooth functions that
	
	\begin{equation*}\label{eq25}
		f\left(\theta_{t+1}\right) \leq f\left(\theta_{t}\right)+\left\langle\nabla f\left(\theta_{t}\right), \theta_{t+1}-\theta_{t}\right\rangle+\frac{L}{2}\left\|\theta_{t+1}-\theta_{t}\right\|^{2},
	\end{equation*}
	%
	When $\cos \left(\theta_{t}, \nabla f\left(\theta_{t}\right)\right)<\delta\eta_{t} / \sqrt{\operatorname{dim}(\theta)}$, let $D = 1 / \sqrt{\operatorname{dim}(\theta)}$. Since $\theta_{t + 1}-\theta_{t} =-\eta_{t} \frac{m_{t}}{{\left(v_{t}+\epsilon\right)^{p}}} + \eta_{t}		\left\langle\hat{\theta}_{t} , \frac{m_{t}}{\left(v_{t}+\epsilon\right)^{p}}\right\rangle \hat{\theta}_{t}$, the above inequality can be further expressed as
	\begin{equation*}\label{eq26}
		\begin{aligned}
			f\left(\theta_{t+1}\right) 
			&\leq 
			f\left(\theta_{t}\right)-\left\langle\nabla f\left(\theta_{t}\right), \eta_{t} \frac{m_{t}}{{\left(v_{t}+\epsilon\right)^{p}}}\right\rangle+\left\langle\nabla f\left(\theta_{t}\right), \eta_{t}\left\langle\hat{\theta}_{t} , \frac{m_{t}}{\left(v_{t}+\epsilon\right)^{p}}\right\rangle \hat{\theta}_{t}\right\rangle \\
			& \quad +
			\frac{\eta_{t}^{2} L}{2}\left\| \frac{g_{t}}{{\left(v_{t}+\epsilon\right)^{p}}}-
			\left\langle\hat{\theta}_{t} , \frac{m_{t}}{\left(v_{t}+\epsilon\right)^{p}}\right\rangle 
			\hat{\theta}_{t}\right\|^{2},
		\end{aligned}
	\end{equation*}
	by taking expectations on both sides of the above inequality with respect to the random variable, we obtain
	\begin{equation*}\label{eq27}
		\begin{aligned}
			\mathbb{E}\left[f\left(\theta_{t+1}\right)\right] \\ &\leq   \mathbb{E}\left[f\left(\theta_{t}\right)\right]-\mathbb{E}\left[\left\langle\nabla f\left(\theta_{t}\right), \eta_{t} \frac{m_{t}}{({v_{t}+\epsilon})^{p}}\right\rangle\right]+
			\mathbb{E}\left[\left\langle\nabla f\left(\theta_{t}\right),\left\langle\hat{\theta}_{t} , \frac{m_{t}}{\left(v_{t}+\epsilon\right)^{p}}\right\rangle \hat{\theta}_{t}\right\rangle\right]\\
			& \quad +\frac{\eta_{t}^{2} L}{2}\mathbb{E}\left[\left\|\frac{g_{t}}{({v_{t}+\epsilon})^{p}}\right\|^{2}\right] 
			+\frac{\eta_{t}^{2} L}{2}\mathbb{E}\left[\left\|\left\langle\hat{\theta}_{t} , \frac{m_{t}}{\left(v_{t}+\epsilon\right)^{p}}\right\rangle \hat{\theta}_{t}\right\|^{2}\right] \\
			&  \leq \mathbb{E}\left[f\left(\theta_{t}\right)\right]-\eta_{t}\mathbb{E}\left[\left\langle\nabla f\left(\theta_{t}\right),  \frac{m_{t}}{({v_{t}+\epsilon})^{p}}\right\rangle\right]+ 
			\eta_{t}\mathbb{E}\left[
			\left\langle\hat{\theta}_{t} , \frac{m_{t}}{\left(v_{t}+\epsilon\right)^{p}}\right\rangle
			\left\langle
			\nabla f\left(\theta_{t}\right), \hat{\theta}_{t}
			\right\rangle
			\right]\\
			&  \quad +\frac{\eta_{t}^{2} L}{2}\mathbb{E}\left[\left\|\frac{g_{t}}{({v_{t}+\epsilon})^{p}}\right\|^{2}\right]
			+\frac{\eta_{t}^{2} L}{2}\mathbb{E}\left[\left\|\left\langle\hat{\theta}_{t} , \frac{m_{t}}{\left(v_{t}+\epsilon\right)^{p}}\right\rangle \hat{\theta}_{t}\right\|^{2}\right]\\
			&  \leq
			\mathbb{E}\left[f\left(\theta_{t}\right)\right]-\eta_{t}\mathbb{E}\left[\left\langle\nabla f\left(\theta_{t}\right), \frac{g_{t}}{({v_{t}+\epsilon})^{p}}\right\rangle\right]+ 
			\eta_{t}\frac{\beta_{1t}}{1-\beta_{1 t}}\mathbb{E}\left[\left\langle\nabla f\left(\theta_{t}\right),  \frac{m_{t}-m_{t-1}}{({v_{t}+\epsilon})^{p}}\right\rangle\right]\\
			&  \quad +
			\eta_{t}\mathbb{E}\left[
			\left\langle\hat{\theta}_{t} , \frac{m_{t}}{\left(v_{t}+\epsilon\right)^{p}}\right\rangle
			\left\langle
			\nabla f\left(\theta_{t}\right), \hat{\theta}_{t}
			\right\rangle
			\right]
			+\frac{\eta_{t}^{2} L}{2}\mathbb{E}\left[\left\|\frac{g_{t}}{({v_{t}+\epsilon})^{p}}\right\|^{2}\right]\\
			& \quad +\frac{\eta_{t}^{2} L}{2}\mathbb{E}\left[\left\|\left\langle\hat{\theta}_{t} , \frac{m_{t}}{\left(v_{t}+\epsilon\right)^{p}}\right\rangle \hat{\theta}_{t}\right\|^{2}\right],
		\end{aligned}
	\end{equation*}
	using Lemma \ref{lem5} and the condition $\cos \left(\theta_{t}, \nabla f\left(\theta_{t}\right)\right)<\delta\eta_{t}D$, the above inequality can be simplified to
	\begin{equation*}\label{eq27}
		\begin{aligned}
			\mathbb{E}\left[f\left(\theta_{t+1}\right)\right]  &\leq 
			\mathbb{E}\left[f\left(\theta_{t}\right)\right]-\eta_{t}\mathbb{E}\left[\left\langle\nabla f\left(\theta_{t}\right), \frac{g_{t}}{({v_{t}+\epsilon})^{p}}\right\rangle\right]+ 
			\eta_{t}\frac{\beta_{1t}}{1-\beta_{1 t}}\frac{2C_{1}^{2}}{\epsilon^{p}}
			+
			\frac{\eta_{t}^{2}\delta C_{1}^{2}D}{\epsilon^{p}}
			+\frac{\eta_{t}^{2} C_{1}^{2}L}{\epsilon^{2p}},
		\end{aligned}
	\end{equation*}
	by rearranging the terms of the above inequality, we get
	\begin{equation*}\label{eq28}
		\begin{aligned}
			\eta_{t}\mathbb{E}\left[\left\langle\nabla f\left(\theta_{t}\right), \frac{g_{t}}{({v_{t}+\epsilon})^{p}}\right\rangle\right]
			&\leq 
			\mathbb{E}\left[f\left(\theta_{t}\right)\right]-\mathbb{E}\left[f\left(\theta_{t+1}\right)\right] + 
			\eta_{t}\frac{\beta_{1t}}{1-\beta_{1 t}}\frac{2C_{1}^{2}}{\epsilon^{p}}
			+
			\frac{\eta_{t}^{2}\delta C_{1}^{2}D}{\epsilon^{p}}
			+\frac{\eta_{t}^{2} C_{1}^{2}L}{\epsilon^{2p}},
		\end{aligned}
	\end{equation*}
	based on Lemma \ref{lem5}, the following inequality can be derived:
	\begin{equation*}\label{eq29}
		\begin{aligned}
			\frac{\eta_{t}}{({C_{1}^{2}+\epsilon})^{p}} \mathbb{E}\left[\left\|g_{t}\right\|^{2}\right]
			&\leq 
			\mathbb{E}\left[f\left(\theta_{t}\right)\right]-\mathbb{E}\left[f\left(\theta_{t+1}\right)\right] + 
			\eta_{t}\frac{\beta_{1t}}{1-\beta_{1 t}}\frac{2C_{1}^{2}}{\epsilon^{p}}
			+
			\frac{\eta_{t}^{2}\delta C_{1}^{2}D}{\epsilon^{p}}
			+\frac{\eta_{t}^{2} C_{1}^{2}L}{\epsilon^{2p}},
		\end{aligned}
	\end{equation*}
	summing both sides of the above inequality from $ t = 1$ to $T $, and considering that $\eta_{t}<1$ and $\beta_{1 t}=\beta_{1} \lambda^{t-1}$, we have
	\begin{equation*}
		\begin{aligned}
			\sum_{t=1}^{T} \frac{\eta_{t}}{({C_{1}^{2}+\epsilon})^{p}} \mathbb{E}\left[\left\|g_{t}\right\|^{2}\right] & \leq \mathbb{E}\left[f\left(\theta_{1}\right)\right]-\mathbb{E}\left[f\left(\theta_{T+1}\right)\right]
			+\frac{2\beta_{1}C_{1}^{2}}{(1-\beta_{1})\epsilon^{p}}\sum_{t=1}^{T} \lambda^{t-1} \\
			& \quad +\left(\frac{\delta C_{1}^{2}D}{\epsilon^{p}}
			+\frac{C_{1}^{2}L}{\epsilon^{2p}}\right) \sum_{t=1}^{T}\eta_{t}^{2} \\
			& \leq 
			\mathbb{E}\left[f\left(\theta_{1}\right)\right]-\mathbb{E}\left[f\left(\theta_{T+1}\right)\right]
			+\frac{2\beta_{1}C_{1}^{2}}{(1-\beta_{1})(1-\lambda)\epsilon^{p}}\\
			&\quad
			+\left(\frac{\delta C_{1}^{2}D}{\epsilon^{p}}
			+\frac{C_{1}^{2}L}{\epsilon^{2p}}\right) \sum_{t=1}^{T}\eta_{t}^{2},
		\end{aligned}
	\end{equation*}
	
	then, we can obtain
	\begin{equation*}
		\begin{aligned}
			\min _{t=1: T} \mathbb{E}\left[\left\|g_{t}\right\|^{2}\right] \sum_{t=1}^{T} \frac{\eta_{t}}{({C_{1}^{2}+\epsilon})^{p}} 
			&\leq \mathbb{E}\left[f\left(\theta_{1}\right)\right]-\mathbb{E}\left[f\left(\theta_{T+1}\right)\right]
			+\frac{2\beta_{1}C_{1}^{2}}{(1-\beta_{1})(1-\lambda)\epsilon^{p}}\\
			& \quad
			+\left(\frac{\delta C_{1}^{2}D}{\epsilon^{p}}
			+\frac{C_{1}^{2}L}{\epsilon^{2p}}\right) \sum_{t=1}^{T}\eta_{t}^{2},
		\end{aligned}
	\end{equation*}
	further, we get
	\begin{equation*}
		\min _{t=1: T} \mathbb{E}\left[\left\|g_{t}\right\|^{2}\right] \leq \frac{f\left(\theta_{1}\right)-f^{*}+\frac{2\beta_{1}C_{1}^{2}}{(1-\beta_{1})(1-\lambda)\epsilon^{p}}
			+\left(\frac{\delta C_{1}^{2}D}{\epsilon^{p}}
			+\frac{C_{1}^{2}L}{\epsilon^{2p}}\right) \sum_{t=1}^{T}\eta_{t}^{2}}{\sum_{t=1}^{T} \frac{\eta_{t}}{({C_{1}^{2}+\epsilon})^{p}}},
	\end{equation*}
	since $\alpha_{t}$ satisfies Assumption \ref{ass33}:
	\begin{equation*}
		\lim _{T \rightarrow \infty} \sum_{t=1}^{T} \eta_{t}^{2} <\infty,
	\end{equation*}
	\begin{equation*}
		\lim _{T \rightarrow \infty} \sum_{t=1}^{T} \frac{\eta_{t}}{({C^{2}+\epsilon})^{p}}=\infty,
	\end{equation*}
	it follows that
	\begin{equation*}
		\lim _{T \rightarrow \infty} \min _{t=1: T} \mathbb{E}\left[\left\|g_{t}\right\|^{2}\right] = 0.
	\end{equation*}
	
	The proof is over.

\end{document}